\documentclass[12pt,psamsfonts]{amsart}
\usepackage{amsmath,amsthm,amsfonts,amssymb,pb-diagram,color}
\usepackage{eucal}
\usepackage{graphicx}
\usepackage{hyperref}
\usepackage[all,knot]{xy}
\xyoption{arc}

\newcommand{\F}{\mathcal{F}}

\newcommand{\A}{\mathcal{A}}
\newcommand{\M}{\mathcal{M}}

\newcommand{\OO}{\mathcal{O}}

\newcommand{\CC}{\mathcal{C}}

\newcommand{\mS}{\mathcal{S}}

\newcommand{\tr}{{\rm tr}}
\newcommand{\Rep}{{\rm Rep}}

\newcommand{\ca}{\mathcal{C}}

\newcommand{\ssl}{\mathfrak{sl}}

\newcommand{\N}{\mathbb{N}}

\DeclareMathOperator{\FPdim}{FPdim}

\DeclareMathOperator{\U}{U}
\newcommand\id{\operatorname{I}}
 
\DeclareMathOperator{\Aut}{Aut}
\DeclareMathOperator{\End}{End}
\DeclareMathOperator{\FP}{FP}
\newcommand{\D}{\mathcal{D}}
\DeclareMathOperator{\Hom}{Hom}

\DeclareMathOperator{\Hilb}{Hilb}
\newcommand{\one}{\mathbf{1}}
\newcommand{\C}{\mathbb C}

\newcommand{\TL}{\mathcal{TL}}

\newcommand{\ot}{\otimes}
\newcommand{\B}{\mathcal{B}}

\newcommand{\ba}{\mathbf{a}}

\numberwithin{equation}{section}

\newtheorem{theorem}{Theorem}[section]

\newtheorem{corollary}[theorem]{Corollary}
\newtheorem{lemma}[theorem]{Lemma}

\newtheorem{conj}[theorem]{Conjecture}

\newtheorem{prop}[theorem]{Proposition}
\theoremstyle{definition}
\newtheorem{nota}[theorem]{Notation}

\newtheorem{remark}[theorem]{Remark}
\newtheorem{remarks}[theorem]{Remarks}

\newtheorem{ex}[theorem]{Example}
\newtheorem{definition}[theorem]{Definition}

\begin{document}
\title[Generalized and quasi-localizations]
{Generalized and quasi-localizations of braid group representations}

\author{C\'esar Galindo}
\address{Departamento de Matem\'aticas\\ Universidad de los Andes\\ Carrera 1 N.
18A - 10,  Bogot\'a, Colombia}
\email{cn.galindo1116@uniandes.edu.co}

\author{Seung-Moon Hong}

\address{Department of Mathematics\\
University of Toledo\\
Toledo, OH 43606-3390}
\email{seungmoon.hong@utoledo.edu}

\author{Eric C. Rowell}
\address{Department of Mathematics\\
    Texas A\&M University \\
    College Station, TX 77843-3368
}
\email{rowell@math.tamu.edu}
\thanks{The last author is partially supported by NSA grant H98230-10-1-0215.}

\begin{abstract}
We develop a theory of localization for braid group representations
associated with objects in braided fusion categories and, more generally, to
Yang-Baxter operators in monoidal categories.  The essential problem is to
determine when a family of braid representations can be uniformly modelled upon
a tensor power of a fixed vector space in such a way that the braid group
generators act ``locally".  Although related to the notion of (quasi-)fiber
functors for
fusion categories, remarkably, such localizations can exist for representations
associated with objects of non-integral dimension.  We conjecture that such
localizations exist precisely when the object in question has dimension the
square-root of an integer and prove several key special cases of the conjecture.
\end{abstract}
\maketitle

\section{Introduction}
Our aim is to generalize and develop the theory of localizations for
braid group representations building upon the groundwork laid in
\cite{RW}.  In this Introduction we summarize the main aspects of
\cite{RW} for the reader's convenience, and then explain the
particular achievements of the current work.  We leave some relevant
standard definitions to later sections for brevity's sake.

The ($n$-strand) \textbf{braid group} $\B_n$ is defined as the group
generated by $\sigma_1,\ldots,\sigma_{n-1}$ satisfying:
\begin{equation}\label{farcommutivity}
\sigma_i \sigma_j=\sigma_j\sigma_i,\;\;\; \textrm{if} \;\;\; |i-j|
\geq 2,
\end{equation}
\begin{equation}\label{braidrelation}
\sigma_i \sigma_{i+1}\sigma_i=\sigma_{i+1} \sigma_{i}\sigma_{i+1}.
\end{equation}

The following definition appears in \cite{RW}, related to the notion of
\emph{towers of algebras} found in \cite{GHJ}:
\begin{definition}\label{seq}
An indexed family of complex $\B_n$-representations $(\rho_n,V_n)$
is a \textbf{sequence of braid representations} if there exist
injective algebra homomorphisms $\tau_n:\C\rho_n(\B_n)\rightarrow
\C\rho_{n+1}(\B_{n+1})$ such that the following diagram commutes:
$$\xymatrix{ \C\B_n\ar[r]\ar@{^{(}->}[d]^\iota &
\C\rho_n(\B_n)\ar@{^{(}->}[d]^{\tau_n} \\ \C\B_{n+1}\ar[r] &
\C\rho_{n+1}(\B_{n+1})}$$
where $\iota:\C\B_n\to\C\B_{n+1}$ is the homomorphism given by
$\iota(\sigma_i)=\sigma_i$.
\end{definition}
Usually the homomorphisms $\tau_n$ will be clear from the context
and will be suppressed. Examples of sequences of braid
representations of interest include those obtained from specialized
quotients of $\C(q)\B_n$ (e.g. Temperley-Lieb algebras
\cite{jones86}) and from objects in braided fusion categories.  An
important role is played by the basic class of examples obtained
from \emph{braided vector spaces} $(W,R)$: that is, a vector space
$W$ and an automorphism $R\in\End(W^{\ot 2})$ satisfying $(\id\ot
R)(R\ot\id)(\id\ot R)=(R\ot\id)(\id\ot R)(R\ot \id)$.  These
sequences of $\B_n$-representations will be denoted
$(\rho^{(W,R)},W^{\ot n})$.

The question considered in \cite{RW} is: when can a given
\emph{unitary} sequence of braid group representations be related to
a sequence of braid group representations of the form
$(\rho^{(W,R)},W^{\ot n})$ in the following sense:

\begin{definition}\label{localdef}
Suppose $(\rho_n,V_n)$ is a sequence of braid representations.  A
\textbf{localization} of $(\rho_n,V_n)$ is a braided vector space
$(W,R)$ such that for all $n\geq 2$ there exist injective algebra homomorphisms
$\phi_n:\C\rho(\B_n)\rightarrow \End(W^{\ot n})$ such that
$\phi_n\circ\rho=\rho^{(W,R)}$.\\
If $(\rho_n,V_n)$ are unitary representations and $R\in\U(W^{\ot 2})$ we
say that $(W,R)$ is a \textbf{unitary localization}.
\end{definition}

The sequences of braid representations studied in \cite{RW} are
constructed as follows: Let $X$ be an object in a braided fusion
category $\CC$.  The braiding $c$ on $\CC$ induces algebra
homomorphisms $\psi_X^n:\C\B_n\rightarrow\End_\CC(X^{\ot n})$ via
$\sigma_i\rightarrow \id_X^{\ot i-1}\ot c_{X,X}\ot \id_X^{\ot
n-i-1}$ (where we have suppressed the associativities for notational
convenience).  The left action of $\End_\CC(X^{\ot n})$ on the
minimal faithful module $W_n^X:=\bigoplus_i \Hom_\CC(X_i,X^{\ot n})$
($X_i$ simple subobjects of $X^{\ot n}$) yields a sequence of braid
representations $(\rho_X,W_n^X)$.  A main result is the following:
\begin{prop}[see \cite{RW} Theorem 4.5]
Suppose that $X$ is a simple object in a braided fusion category $\CC$, the
sequence $(\rho_X,W_n^X)$ is localizable
and $\psi_X^n$ is surjective (so that $\Hom_\CC(X_i,X^{\ot n})$ is
irreducible as a $\B_n$ representation for each $i$).  Then
$\FPdim(X)^2\in\N$.
\end{prop}
Otherwise stated, the hypotheses imply that the fusion subcategory
generated by $X$ is \emph{weakly integral}.

On the other hand, the sequence of braid group representations $(\rho_X,W_n^X)$
associated with an object $X\in\Rep(H)$ for a finite dimensional semisimple
quasi-triangular Hopf algebra $H$ is easily seen to be localizable (see Prop.
\ref{localization of quasi-Hopf and Hopf} for a stronger statement).  Notice
that $\Rep(H)$ is an \textit{integral} fusion category in this case (\emph{i.e.}
$\FPdim(X)\in\N$ for all $X\in\Rep(H)$).
The main conjecture in \cite{RW} is the following:
\begin{conj}[cf. \cite{RW} Conjecture 4.1]
Let $X\in\CC$ be a simple object in a braided fusion category.  Then
$(\rho_X,W_n^X)$ is localizable if, and only if, $\FPdim(X)^2\in\N$.
\end{conj}

A related conjecture is the following:
\begin{conj}[cf. \cite{RSW} Conjecture 6.6]
Let $X$ be a simple object in a braided fusion category.  Then the image of
$\B_n$ under $(\rho_X,W_n^X)$ is a finite group if, and only if,
$\FPdim(X)^2\in\N$.
\end{conj}

The following brings these ideas full circle:
\begin{conj}[cf. \cite{RW} Conjecture 3.1]
Suppose $(V,R)$ is a braided vector space such that $R$ is \emph{unitary} and
\emph{finite order}.  Then the image of the $\B_n$ representation on $V^{\ot n}$
defined by $\sigma_i\rightarrow \id_V^{\ot i-1}\ot R\ot \id_V^{\ot n-i-1}$ has
finite image.
\end{conj}
It should be noted that neither the unitarity nor finite order condition can be
dropped.

The following is a summary of the current work:
In Section \ref{gen and quasi YBO} we generalize the notion of Yang-Baxter
operators in two ways, leading to \textit{quasi-} and
\textit{$(k,m)$-generalized} braided vector spaces.
We give a more flexible version of Definition \ref{seq} in terms of
sequences of algebras equipped with braid group representations (Definition
\ref{seq of algebras}).
In Section \ref{locs} we define $\CC$-localizations over arbitrary
monoidal categories $\CC$, so that a Vec$_f$-localization is the same as
Definition \ref{localdef}.  Moreover, when $\CC$ is a monoidal category
associated with a quasi-braided vector space we obtain the notion of a
\emph{quasi-localization} and prove that modules over quasitriangular
\emph{quasi}-Hopf algebras lead to \emph{quasi}-localizable sequences of
algebras just as modules over quasitriangular Hopf algebras lead to localizable
sequences of algebras.  As a by-product we obtain a criterion for the existence
of a fiber functor for an integral braided fusion category. To continue the
analogy with reconstruction-type theorems we describe \emph{weak} localizations
as well.  A variant of localization associated with $(k,m)$-generalized braided
vector spaces is also given, which, although somewhat mysterious, is more
explicit than quasi-localizations.
The main result of the somewhat technical Section \ref{unitarity of wgt}
is that the braid group representations associated with any weakly
group-theoretical braided fusion category $\CC$ are unitarizable, and if in
addition $\CC$ is integral a unitary (quasi-)localization exists.  We also give
an example associated with quantum $\ssl_3$ at $6$th roots of unity which
illustrates the differences between the various forms of localization studied
here.
In Section \ref{conj and results} we state a version of the conjectures
mentioned above for quasi- and generalized localizations and provide evidence.
In particular we prove the statement analogous to \cite[Theorem 4.5]{RW} (see
above) in these settings.

\textbf{Acknowlegements}  E.R. thanks D. Naidu and Z. Wang for useful comments.

\section{Preliminaries}\label{prelims}
In this section we recall some standard notions in order to establish notation
and conventions.  Much of the material here can be found in \cite{BK}, but we
typically adopt the conventions of \cite{ENO}.
\subsection{$k$-linear categories}

Let $k$ be a field. A $k$-linear category $\CC$ is a category in which the
Hom-sets are $k$-vector spaces, the compositions are $k$-bilinear
(we do not assume the existence of direct sums or zero object, so $\CC$ may not
be additive). The
notion of a $k$-linear functor $\CC\to \D$, and a $k$-bilinear
bifunctor $\CC\times \CC'\to \D$ for $k$-linear categories $\CC,
\CC', \D,$ will be obvious.

A $k$-linear category $\CC$ is said to be of \textbf{locally finite dimension}
if $\Hom_\CC(X,Y)$ is of finite dimension for any $X, Y\in \CC$.  In this paper
we
shall only consider $k$-linear categories of locally
finite dimension.


\subsection{$C^*$-categories}
Most of the material here can be found in \cite{Mueg}.
\begin{definition}
A $\mathbb C$-linear category $\D$ is called a \textbf{complex
$*$-category} if:

\begin{enumerate}
  \item There is an involutive antilinear contravariant endofunctor $*$ of
$\D$ which is the identity on objects. The image of $f$ under $*$ will be
denoted by $f^*$.
  \item For each $f\in  \Hom_\D(X, Y)$,  $f^*f = 0$ implies $f = 0$.
\end{enumerate}
 In particular, in a complex $*$-category, each $\Hom_\D(X, X)$ is a
$*$-algebra with identity.
A \textbf{$C^*$-category} $\D$ is a complex $*$-category such that the
spaces $\Hom_\D(X,Y)$ are Banach spaces and the norms satisfy $$|| f
g||\leq ||f||\ ||g||,  \   \  ||f^*f||=||f||^2,$$ for all  $f \in
\Hom_D(X, Y), g \in\Hom_\D(Y,Z )$. (Then the algebras $\Hom_\D(X,
X)$ are $C^*$-algebras.)

\end{definition}
\begin{remark}
Every abelian complex $*$-category of locally finite dimension
admits a unique structure of $C^*$-category (see \cite[Proposition
2.1]{Mueg}). Since we are only interested in categories of locally
finite dimension, for us abelian $C^*$-category and abelian complex
$*$-category are equivalent notions.
\end{remark}
\begin{ex}
Let $R$ be a finite dimensional $C^*$-algebra, then
 the category $\mathcal{U}$-$\Rep(R)$ of $*$-representations on
finite dimensional Hilbert spaces is an abelian $*$-category and thus admits a
unique $C^*$-structure.  Note
$\Rep(R)$ is equivalent to $\mathcal{U}$-$\Rep(R)$.
\end{ex}
Let $X$ and $Y$ be objects in a $*$-category. A morphism $u:X\to Y$ is
\textbf{unitary} if $uu^*=\id_Y$ and  $u^*u=\id_X$. A morphism
$a:X\to X$ is \textbf{self-adjoint} if $a^*=a$.

A natural transformation $\gamma:F\to G$, between functors $F,
G:\D_1\to \D_2$ with $\D_2$ a $*$-category is called
\textbf{unitary natural transformation} if $\gamma_X$ is unitary for
each $X\in \D_1$.
\begin{remark}\label{remark polar}
Let $\D$ a $*$-category
\begin{enumerate}
  \item The  opposite category $\D^{op}$ is a $*$-category with the same
$*$-structure.
  \item  Every isomorphism in a $C^*$-category has a polar decomposition,
\emph{i.e.}, if $f:X\to Y$ is an isomorphism, then $f = ua$ where
$a: X \to X$ is self-adjoint and $u: X \to Y$ is unitary, see
\cite[Proposition 8]{Baez}.
\end{enumerate}
\end{remark}

\begin{definition}
A \textbf{$*$-functor} $F:\D \to \D'$ between $*$-categories $\D$
and $\D'$ is $\mathbb C$-linear functor such that $F(f^*)=F(f)^*$
for all $f\in \Hom_\D(X,Y)$.
\end{definition}

\begin{remark}
Let $R$ be a finite dimensional $C^*$-algebra, then every exact
endo\-functor $F:\mathcal{U}$-$\Rep(R)\to \mathcal{U}$-$\Rep(R)$ is
naturally equivalent to a functor of the form $M\otimes_R (?)$,
where $M\in $Bimod$(R)$ is an $R$-bimodule. The functor $M\otimes_R
(?)$ is a $*$-functor if and only if $M$ is a unitary $R$-bimodule
or equivalently a unitary $R\otimes R^{op}$-module. We shall denote
the $*$-category of unitary $R$-bimodules as
$\mathcal{U}$-Bimod$(R)$.
\end{remark}

Two $*$-categories $\D$ and $\D'$  are  \textbf{unitarily
equivalent} if there exist $*$-functors $F:\D\to \D'$ and $G:\D'\to
\D$ and unitary natural isomorphisms $\id_{\D} \to G\circ F$,
$\id_{\D'} \to F\circ G$.

\begin{remark}
\begin{enumerate}
\item Every equivalence of $*$-category is equivalent to a unitary equivalence.

\item Every abelian $*$-category is unitary equivalent to a category
$\bigoplus_{i\in I}\Hilb_f$, a direct sum of copies of the category of
finite-dimensional Hilbert spaces.

\end{enumerate}

\end{remark}

Given $*$-categories $\D_1$ and $\D_2$ we define the $\mathbb
C$-linear category $\Hom^*(\D_1,\D_2)$, where objects are
$*$-functors $F:\D_1\to \D_2$ such that $F(X)\neq 0$ for finite
isomorphism classes of simple objects, and morphisms are natural
transformations. The category $\Hom^*(\D_1,\D_2)$ has a natural
structure of $*$-category with $*$-structure
$(\alpha^*)_X=(\alpha_X)^*$.

\subsection{Unitary fusion categories and module categories}

For us, a monoidal category $(\CC, \otimes, \alpha, 1, \lambda,\rho)$ will
always be
$k$-linear with associativity constraint $\alpha_{V,W,Z} : (V \otimes W)
\otimes  Z\to  V \otimes (W \otimes Z)$, unit object $1$ and unit constraints
$\lambda,\rho$ satisfying the usual (triangle and pentagon) axioms.  Note that
we do not assume that $1$ is a simple object.  As is customary, we will assume
that the unit constraints are identities and abuse notation by referring to
``the monoidal category $(\CC,\ot,\alpha)$.''

\begin{definition} A \textbf{multi-fusion category} is a monoidal, rigid,
semisimple category with a finite number of isomorphism classes of simple
objects.  A \textbf{fusion category} is a multi-fusion category in which
$1$ is a simple object.

A \textbf{unitary (multi)-fusion category} is a (multi)-fusion
category $(\CC,\otimes, \alpha)$, where $\CC$ is a
positive $*$-category, the constraints are unitary natural
transformations, and $(f\otimes g)^*= f^*\otimes g^*$, for every pair of
morphisms
$f,g$ in $\CC$.
\end{definition}
\begin{ex}
\begin{enumerate}
  \item The category of finite dimensional Hilbert spaces $\Hilb_f$, with the
  tensor product of Hilbert spaces is a unitary fusion category.
  \item If $R$ is a finite dimensional $C^*$-algebra, then
  $\mathcal{U}$-Bimod$(R)$ is a unitary multi-fusion category.
  \item Recall that a finite dimensional (quasi) Kac  algebra is a
(quasi) Hopf algebra $H$, such that $H$ is a $C^*$-algebra, $\Delta$
and $\varepsilon$ are $*$-algebras morphisms, and if $H$ is a
quasi-Hopf algebra the associator must satisfy $\Phi^*=\Phi^{-1}$.
In this case the category of unitary $H$-modules is a unitary fusion
category.

\end{enumerate}

\end{ex}

A \textbf{$*$-monoidal functor} between unitary fusion categories is a monoidal
functor $(F,F^0, F^1):\CC_1\to\CC_2$, such that $F$ is a
$*$-functor, and $F^0_{X,Y}$, $F^1_X$ are  unitary natural
transformations.

Module categories over monoidal categories are defined in \cite{Ost}.
\begin{definition}
Let $\CC$ be a unitary fusion category.  A left \textbf{$\CC$-module
$*$-category} is a left
$\CC$-module category $(\M,\overline{\otimes},\mu)$ such that $\M$ is a
$*$-category, the constraints are unitary natural transformations,
and $(f\overline{\otimes} g)^* =f^*\overline{\otimes} g^*$ for all
 $f\in \CC, g\in \M$.\end{definition}

\begin{definition}
A \textbf{$\CC$-module $*$-functor}  $F: \M\to \mathcal{N}$ between
$\CC$-module $*$-categories $\M, \mathcal{N}$,  is $\CC$-module
functor $(F,F^0,F^1)$, such that $F$ is a $*$-functor, and $F^0,
F^1$ are unitary natural transformations.
\end{definition}

Recall that a monoidal category is called strict if the
associativity constraint is the identity. A module category
$(\M,\overline{\otimes}, \mu)$ over a strict monoidal category is
called a \textbf{strict module category}  if the constraint $\mu$ is
the identity.  Using the same argument as in \cite[Proposition
2.2]{Ga1} we may assume that every unitary fusion category $\CC$ and
every $\CC$-module $*$-category is strict.

If $\alpha:F\to G$ is a module natural transformation between
$\CC$-module $*$-functors,  $\alpha^*:G\to F$ is a $\CC$-module
natural transformation. Thus the category
$\Hom_\CC^*(\M,\mathcal{N})$ of all $\CC$-module $*$-functors and
$\CC$-module natural transformations has a $*$-structure.

\section{Quasi- and Generalized Yang-Baxter operators}\label{gen and quasi YBO}

\subsection{Yang-Baxter operators}

We shall recall the definition of Yang-Baxter operator on a monoidal
category, see \cite{Kas}:

\begin{definition}
 If $V$ is an object of a monoidal category $(\CC,\otimes, \alpha)$ and
$c\in\Aut(V\otimes V)$ satisfies the equation
\begin{equation} \label{YBoper}
\alpha_{V,V,V}(c\otimes \id)\alpha_{V,V,V}^{-1}(\id\otimes
c)\alpha_{V,V,V}(c\otimes \id)=\\
(\id\otimes c)\alpha_{V,V,V}(c\otimes
\id)\alpha_{V,V,V}^{-1}(\id\otimes c) \alpha_{V,V,V}
\end{equation}
(where $I=I_V$) then $c$
is called a \textbf{Yang-Baxter operator} on $V$.
\end{definition}

Yang-Baxter operators define
representations of the braid groups in the following way: define
$V^{\circledast 1}=V$, $V^{\circledast n}=V^{\circledast
(n-1)}\otimes V$ that is, all left parentheses appear left of the
first $V$ with the same convention for tensor products of morphisms.
Define automorphisms $c_1, \ldots, c_{n-1}$ of $V^{\circledast  n}$
by
\begin{equation}\label{rep from YBO}
 c_{i}=(\alpha_{V^{\circledast (i-1)},V,V}^{-1}\otimes \id_V^{\ot(n-i-1)})
 (\id_{V^{\circledast (i-1)}}\otimes c\otimes \id_{V}^{\ot(n-i-1)})
 (\alpha_{V^{\circledast (i-1)},V,V}\otimes\id_V^{\ot(n-i-1)}),
\end{equation}
where, for example, $$\alpha_{V^{\circledast 3},V,V}\ot \id_V^{\ot
2}=(\alpha_{V^{\circledast 3},V,V}\ot \id_V)\ot
\id_V\in\Hom_\CC(V^{\circledast 7},((V^{\circledast 3}\ot (V\ot
V))\ot V)\ot V).$$ Thus, for any $n$ the map
\begin{gather}\label{morphism YB op}
\begin{split}
    \rho_n: \mathcal{B}_n &\to \Aut_\CC(V^{\circledast  n})\\
    \sigma_i &\mapsto c_i
\end{split}
\end{gather}
is a group homomorphism (see \cite[Lemma XV.4.1]{Kas}).

\begin{definition}
A Yang-Baxter operator $c$ on a finite dimensional vector space
$V\in$ Vec$_f$ is called a \textbf{braided vector space} $(V,c)$. A
\textbf{unitary} braided vector space $(\mathcal{H},c)$ is a unitary
Yang-Baxter operator $c$ on an object $\mathcal{H}$ in the category
of finite dimensional Hilbert spaces.\end{definition}

\begin{remark}\label{reps from YB operators}
The group $\Aut_\CC(V^{\circledast n})$ is linear: it has a faithful
action on the vector space $\End_\CC(V^{\circledast n})$.  Pulling
back via $\rho_n$ we obtain a linear representation of $\B_n$ on
$\End_\CC(V^{\circledast n})$.

In the special case of  braided vector spaces $(V,c)$ (\emph{i.e.}
$\CC=$Vec$_f$)  one has $\Aut_\CC(V^{\circledast
n})=GL(V^{\circledast n})$ so that one obtains a linear
representation of $\B_n$ on $V^{\circledast n}$. Moreover, the
associativity isomorphisms for Vec$_f$ are: $\alpha_{V,V,V}:(v_1\ot
v_2)\ot v_3\rightarrow v_1\ot (v_2\ot v_3)$ so that, with respect to
the obvious compatible choices of bases, the $\alpha_{V^{\circledast
(i-1)},V,V}$ are all represented by the identity matrix.  Thus the
$\B_n$-representation on $V^{\circledast n}$ obtained from $\rho_n$
is equivalent to a \emph{matrix} representation of the form
$$\sigma_i\rightarrow I^{\ot i-1}\ot c\ot I^{\ot n-i-1}.$$
In particular, our definition of braided vector space is the same as
that of \cite{AS}, \emph{i.e.} a pair $(V,c)$ where $c\in\Aut(V\ot
V)$ satisfies (\ref{YBoper}) on $V^{\ot 3}$ with the $\alpha$
removed.
\end{remark}

\subsection{Quasi Yang-Baxter operators}

Let $(\D,\otimes)$ be a monoidal category and $A\in\D$ an object. We
shall denote by $\langle A \rangle$ the
monoidal subcategory of $\D$ generated by $A$, that is, the objects
of $\langle A\rangle$ are isomorphism classes of $A^{\ot n}$ for
$n\geq 0$ ($A^{\ot 0}=\one$).

Let $a=\{a_{X,Y}\}_{X,Y\in \langle A\rangle}$ be a family of
isomorphisms $a_{X,Y}: (X\otimes A)\otimes Y \to X\otimes(A\otimes
Y)$, $X,Y\in \langle A\rangle$. We shall say that $a$ is
\textbf{self-natural} if $a_{X,Y}$ is natural in $X$ and $Y$ for
every pair of isomorphisms constructed as a composition of tensor
products of identities, elements and inverses of elements in $a$.

If $a=\{a_{X,Y}\}_{X,Y\in \langle A\rangle}$ is a self-natural
family of isomorphisms we define $a^2=\{a_{X,Y}^2\}_{X,Y\in \langle
A\rangle}$ a new family of self-natural on $A\otimes A$ by the
commutativity of pentagonal diagram:

$$
\xygraph{ !{0;/r4.5pc/:;/u4.5pc/::}[]*+{((X\otimes A)\otimes
A)\otimes Y} (
  :[u(1.1)r(1.7)]*+{(X\otimes A)\otimes(A\otimes Y)} ^{a_{X\otimes A, Y}}
  :[d(1.1)r(1.7)]*+{X\otimes(A\otimes (A\otimes Y))}="r" ^{a_{X,A\otimes Y}}
  ,
  :[r(.6)d(1.5)]*+!R(.3){(X\otimes(A\otimes A))\otimes Y} ^{a_{X,A}\otimes
\id_Y}
  :[r(2.2)]*+!L(.3){X\otimes((A\otimes A)\otimes Y)} ^{a_{X,Y}^2}
  : "r" ^{\id_X\otimes a_{A,Y}}
) }$$



\begin{definition}
Let $\D$ be a monoidal category and $A$ an object in $\D$. A
\textbf{quasi-Yang-Baxter operator}  on $A$ is a pair $(a,c)$, where
$a=\{a_{X,Y}\}_{X,Y\in \langle A\rangle}$ is a family of
self-natural transformations and $c:A\otimes A\to A\otimes A$ is an
automorphism such that

\begin{equation}\label{axi 1 quasi}
a_{A,A}(c\otimes \id)a_{A,A}^{-1}(\id\otimes c)a_{A,A}(c\otimes
\id)=(\id\otimes c)a_{A,A}(c\otimes \id)a_{A,A}^{-1}(\id\otimes c)
a_{A,A},
\end{equation} and the diagram
\begin{equation}\label{axi 2 quasi}
\xymatrix{ (X\otimes(A\otimes A))\otimes Y \ar[d]^{(\id_X\otimes
c)\otimes \id_Y}\ar[r]^{a_{X,Y}^2} & X\otimes ((A\otimes A)\otimes
Y) \ar[d]^{\id_X\otimes (c\otimes \id_Y)}
\\
(X\otimes(A\otimes A))\otimes Y \ar[r]^{a_{X,Y}^2} & X\otimes
((A\otimes A)\otimes Y) }
\end{equation}commutes for all $X,Y\in \langle
A\rangle$.

If $\D$ is a $C^*$-tensor category and $A$ is an object, a \textbf{unitary
quasi-Yang-Baxter operator} on $A$ is a quasi-Yang-Baxter operator
$(a,c)$ such that $c$ and $a_{X,Y}$ are unitary isomorphisms for all
$X,Y\in \langle A\rangle$.
\end{definition}

\begin{ex}
\item Quasi-Yang-Baxter operators generalize the notion of Yang-Baxter operators
over monoidal categories. In fact, if $(A,c)$ is a Yang-Baxter
operator over a monoidal category $(\CC,\otimes   ,\alpha)$ then
$a_{X,Y}=\alpha_{X,A,Y}$ defines a self-natural family of
isomorphisms such that $a^2_{X,Y}=\alpha_{X,A\otimes A, Y}$ (by the
pentagonal identity), and by the naturality of  $\alpha$, the
diagram \ref{axi 2 quasi} commutes, so $(\alpha_{X,A,Y},c)$ is
a quasi-Yang-Baxter operator on $A$.
\end{ex}

\begin{definition}
A  \textbf{quasi-braided vector space} is a  quasi-Yang-Baxter operator over
$\text{Vec}_f$. A \textbf{unitary} quasi-braided vector space is a unitary
quasi-Yang-Baxter operator in the category of finite dimensional Hilbert
spaces.
\end{definition}

\begin{ex}\label{exam quasi hopf quasi braid}
Let $(H,\Phi,R)$ be a quasi-triangular quasi-Hopf algebra and $V$ an
$H$-module.  Then for each pair $X,Y \in \langle V\rangle$ the family
of isomorphisms $a_{X,Y}((x\otimes v)\otimes y)=\Phi(x\otimes (v\otimes
y))$ is self-natural and $c:V\otimes V\to V\otimes V$ given by $c(v\otimes v')=
(R(v\otimes v'))_{21}$  defines a quasi-braided vector space structure on
$V\in\text{Vec}_f$.
(recall that $(v\otimes w)_{21}=w\otimes v$).  Observe that $c$ is a Yang-Baxter
operator on $V\in\Rep(H)$, but cannot be a braided vector space unless the
forgetful functor is a fiber functor.
\end{ex}

Let $(a,c)$ be a quasi-Yang-Baxter operator on $A\in\CC$. We define
$A^{\circledast 1}=A$,
$A^{\circledast n}=A^{\circledast (n-1)}\otimes A$ and automorphisms $c_1,
\ldots,
c_{n-1}$ of $A^{\circledast  n}$ by
\begin{equation}
 c_{i}=(a^{-1}_{A^{\circledast (i-1)},A}\otimes \id_A^{\ot(n-i-1)})
 (\id_{A^{\circledast (i-1)}}\otimes c\otimes \id_A^{\ot(n-i-1)})
 (a_{A^{\circledast (i-1)},A}\otimes\id_A^{\ot(n-i-1)}).
\end{equation}
Now suppose that $\CC$ is strict.  We define a monoidal category
$\overline{(A,a,c)}$ where Obj$\overline{(A,a,c)}=$Obj$\langle A
\rangle$ and the tensor product of objects is the same as in $\CC$.
We define inductively self-natural families of isomorphisms
$a^n_{X,Y}:(X\otimes A^{\circledast n})\otimes Y \to
X\otimes(A^{\circledast  n}\otimes Y)$ by $a^{1}=a$ and $a^n_{X,Y}$
by the commutativity of the diagram

$$
\xygraph{ !{0;/r4.5pc/:;/u4.5pc/::}[]*+{((X\otimes A^{\circledast
(n-1)})\otimes A)\otimes Y} (
  :[u(1.1)r(1.7)]*+{(X\otimes A^{\circledast (n-1)})\otimes(A\otimes Y)}
^{a_{X\otimes A^{\circledast (n-1)}, Y}}
  :[d(1.1)r(1.7)]*+{X\otimes(A^{\circledast (n-1)}\otimes (A\otimes Y))}="r"
^{a_{X,A\otimes Y}^{n-1}}
  ,
  :[r(.6)d(1.5)]*+!R(.3){(X\otimes(A^{\circledast (n-1)}\otimes A))\otimes Y}
^{a_{X,A}^{n-1}\otimes \id_Y}
  :[r(2.2)]*+!L(.3){X\otimes((A^{\circledast (n-1)}\otimes A)\otimes Y)}
^{a_{X,Y}^n}
  : "r" ^{\id_X\otimes a_{A^{\circledast (n-1)},Y}}
) }$$

A morphism $f:A^{\circledast  n}\to A^{\circledast  m}$ in
$\overline{(A,a,c)}$ is a morphism in $\CC$ such that the diagram
$$\xymatrix{ (X\otimes A^{\circledast  n})\otimes Y \ar[d]^{(\id_X\otimes
f)\otimes
\id_Y}\ar[r]^{a^n_{X,Y}} & X\otimes (A^{\circledast  n}\otimes Y)
\ar[d]^{\id_X\otimes (f\otimes \id_Y)}
\\
(X\otimes A^{\circledast  m})\otimes Y \ar[r]^{a^m_{X,Y}} & X\otimes
(A^{\circledast m}\otimes Y) }$$ is commutative for any $X,Y\in
\overline{(A,a,c)}$.

Using the same arguments of \cite[Section 2.1]{Dav}, we can prove
that $\overline{(A,a,c)}$ is a monoidal category with tensor product
$\otimes$ and associativity constraint $\alpha_{A^{\circledast
m},A^{\circledast n},A^{\circledast s}}=a_{A^{\circledast
m},A^{\circledast s}}^n$.

\begin{prop}
Let $(a,c)$ be a quasi-Yang-Baxter operator on $A\in\CC$. There
exists a unique homomorphism of groups $\rho_n: \mathcal{B}_n\to
\Aut_\CC(A^{\circledast  n})$ sending the generator $\sigma_i$ of
$\mathcal{B}_n$ to $c_i$ for each $i=1,\ldots, n-1$.
\end{prop}
\begin{proof} We may assume without loss of generality that
$\CC$ is strict, since in the non-strict case by the coherence of $\CC$ every
object in
$\langle A\rangle$ is isomorphic to $A^{\circledast  n}$ by a unique
isomorphism constructed as composition and tensor products of the
constraint isomorphisms and their inverses.  By
hypothesis $c:A\otimes A\to A\otimes A$ is a Yang-Baxter operator in
$\overline{(A,a,c)}$, so the proposition follows from the group
morphisms (\ref{morphism YB op}).
\end{proof}

\subsection{Generalized Yang-Baxter Operators}
In this subsection we will describe a second approach to generalizing
Yang-Baxter operators based upon the generalized Yang-Baxter equation introduced
in \cite{RZWG}.

\begin{definition} Let $V$ be an object in a monoidal category
$(\CC,\ot,\alpha)$ and $k,m\in\N$ with $k>m$.  Set $I_m:=I_{V^{\circledast m}}$
and denote by $\alpha_{k,m}$ the natural isomorphism $V^{\circledast k}\ot
V^{\circledast m}\rightarrow V^{\circledast m}\ot V^{\circledast k}$ coming from
$\alpha$. For an automorphism $c\in\Aut(V^{\circledast  k})$ the
\textbf{$(k,m)$-generalized Yang-Baxter equation ($(k,m)$-gYBE}) on $V$ is
 \begin{equation}\label{gYBE}
  \alpha_{k,m}(c\ot I_m)\alpha^{-1}_{k,m}(I_m\ot c)\alpha_{k,m}(c\ot
I_m)=(I_m\ot c)\alpha_{k,m}(c\ot I_m)\alpha^{-1}_{k,m}(I_m\ot c)\alpha_{k,m}.
 \end{equation}

A solution $c$ to the $(k,m)$-gYBE is called a \textbf{$(k,m)$-gYB operator} on
$V$ if in addition $c$ satisfies the following equations for all $4\leq j$:
\begin{equation}
\label{gYBfarcomm}
 \alpha_{j}(c\ot I_{m(j-2)})\alpha^{-1}_{j}(I_{m(j-2)}\ot c)=(I_{m(j-2)}\ot
c)\alpha_{j}(c\ot I_{m(j-2)})\alpha_{j}^{-1}.
\end{equation} where $\alpha_j=\alpha_{k,m(j-2)}$.
\end{definition}

\begin{remark}
\begin{enumerate}
\item To verify that a solution $c$ to the $(k,m)$-gYBE is a $(k,m)$-gYB operator it is sufficient
to check eqn. (\ref{gYBfarcomm}) for $j<\frac{k}{m}+2$ since for $j\geq \frac{k}{m}+2$ eqn. the
operators $c_1$ and $c_j$ act nontrivially on disjoint tensor factors so (\ref{gYBfarcomm}) is automatic.
\item The requirement $k>m$ is to avoid trivialities in what follows: see
Prop. \ref{km-rep} below.
 \item A Yang-Baxter operator on $V$ in a category $\CC$ is a $(2,1)$-gYB
operator.
\item One may also define generalized \emph{quasi}-Yang-Baxter operators but we
shall have no need to do so.
\end{enumerate}
\end{remark}

\begin{definition}\label{generalized BVS}
A (\textbf{unitary}) \textbf{$(k,m)$-generalized braided vector
space} is a (unitary) $(k,m)$-gYB operator $c$ on a finite dimensional
(Hilbert) vector space $V\in$ Vec$_f$.
\end{definition}

As in Remark \ref{reps from YB operators}, in the case of a $(k,m)$-generalized
braided vector space $(V,c)$ we may suppress the $\alpha$ in eqns. (\ref{gYBE})
and (\ref{gYBfarcomm}) by choosing bases appropriately.
Define operators
\begin{equation}\label{rep from gYB oper}
 c^{k,m}_i=I_m^{\ot i-1}\ot c\ot I_m^{\ot n-i-1}
\end{equation}
for each $1\leq i\leq n-1$.
We have the following:
\begin{prop}\label{km-rep}
Suppose $(V,c)$ is a $(k,m)$-generalized braided vector space. The assignment
$\sigma_i\rightarrow c^{k,m}_i$ defines a group homomorphism for $n\geq 1$:
$$\rho^c:\B_n\rightarrow \Aut_\CC(V^{\ot k+m(n-2)}).$$
\end{prop}
\begin{proof}
Eqn. (\ref{gYBE}) implies that
$c^{k,m}_1c^{k,m}_2c^{k,m}_1=c^{k,m}_2c^{k,m}_1c^{k,m}_2$ and it
follows that $c^{k,m}_i$ and $c^{k,m}_{i+1}$ satisfy relation
(\ref{braidrelation}). For relation (\ref{farcommutivity}) it is
enough to check that $c^{k,m}_1$ commutes with $c^{k,m}_3,\ldots,
c^{k,m}_{n-1}$, which is  eqn. (\ref{gYBfarcomm}).
\end{proof}

\begin{remark}
 A result analogous to Prop. \ref{km-rep} for $(k,m)$-gYB operators on objects
in arbitrary monoidal categories $\CC$ can be proved with essentially no change
except to eqn. (\ref{rep from gYB oper}).
\end{remark}

\begin{ex} Set $J=\begin{pmatrix}
0&0&0&1\\0&0&1&0\\0&1&0&0\\1&0&0&0\end{pmatrix}$.  In \cite{RZWG} it is shown
that
the matrix
$R=\frac{1}{\sqrt{2}}\begin{pmatrix} I & J\\-J&I\end{pmatrix}$
 yields a unitary $(3,2)$-generalized braided vector space $(\C^2,R)$ (where $I$
represents the $4\times 4$ identity matrix).
\end{ex}

\section{Localizations}\label{locs}

In this section we give some variations on the notion of localization in
somewhat more flexible settings.  Although the categories we have in mind are
typically highly structured (abelian, semisimple etc.) and confer significant
structure on the associated braid representations, we discard as much of these
restrictions as is possible.

\subsection{Sequences of algebras under $\C\B$}

\begin{definition}
A $\C$-category (resp. a $C^*$-category) $\A$ shall be called
\textbf{diagonal} if:
\begin{enumerate}
  \item Obj$(\A)=\mathbb N$ (we shall denote by $[n]$ the object
  corresponding to the natural $n$),
  \item for all $n, m\in \N, m\neq n$, $\Hom_\A([m],[n])=0$.
\end{enumerate}

\end{definition}

To specify a diagonal category $\A$ it is enough to
describe the algebras $\End_\A([n])$ for each $n$. We shall denote
by $\C\B$ the diagonal $\C$-linear $*$-category where
$\End_{\C\B}([n])=\C\B_n$ for all $n\in \N$.  Observe that this is
just the linearization of the category $Braid$ that appears in
\cite{JS} (as a $*$-category).

\begin{definition}
\begin{enumerate}
  \item Let $\A$ and $\D$ be diagonal categories. A morphism $\F:\A\to \D$
from $\A$ to $\D$ is a $\C$-linear functor $F:\A\to \D$ such that
$\F([n])=[n]$ for all $n\in \N$.

\item A \textbf{diagonal category under $\C\B$} is a pair $(\A,\rho_\A)$ where
$\A$ is a diagonal category and $\rho_\A:\C\B\to \A$ is a morphism of
diagonal categories.

\end{enumerate}
\end{definition}
To specify a morphism $\F:\A\to\D$ of
diagonal categories one need only describe
$\F:\End_\A([n])\to\End_\D([n])$ for each $n$, so that a diagonal
category under $\C\B$ is a family of $\C$-algebra representations
$\rho_\A:\C\B_n\to\End_\A([n])$.

In \cite{TurWen97} tensor representations $F:Tang\rightarrow$ Vec$_f$ of the
tangle category $Tang$ are defined as covariant tensor functors such that
$F([n])=V^{\ot n}$ for some $V\in$Vec$_f$.  Since $\C\B$ is a subcategory of
(the $\C$-linearization of) $Tang$ such a tensor representation gives rise to a
diagonal category under $\C\B$.

\begin{nota} If $(\A,\rho_\A)$ is a diagonal category under $\C\B$, we shall
denote by $(\underline{\A},\rho_\A)$ the diagonal category under
$\C\B$ given by $\End_{\underline{\A}}([n]):= \rho_\A(\C\B_n)\subset
\End_\A([n])$, and we shall denote by $\underline{\A}_{+1}$, the
diagonal category under $\C\B$ where
$\Hom_{\underline{\A}_{+1}}([n],[m])=\Hom_{\underline{\A}}([n+1],[m+1])$,
for all $n\in \N$.
\end{nota}

\begin{definition}
Let $\A$ be a diagonal category.  A \textbf{representation of $\A$} is an
indexed family $(\vartheta_n,V_n)$ ($n\geq 1$) of representations $\vartheta_n:
\End_\A([n])\rightarrow \End_\C(V_n)$.
\end{definition}
If $(\vartheta_n,V_n)$ is a representation of a diagonal category $(\A,\rho_\A)$
under $\C\B$ then we obtain another diagonal category
$(\vartheta(\A),\vartheta\circ \rho_\A)$ under $\C\B$ by setting
$\End_{\vartheta(\A)}([n]):=\End_\C(V_n)$ and
$$\vartheta(\rho_\A(\C\B_n)):=\vartheta_n\circ\rho_\A(\C\B_n)\subset\End_\C(V_n)
=\End_{\vartheta(\A)}([n]).$$

The following is a (realization-independent) replacement of Definition
\ref{seq}:
\begin{definition}\label{seq of algebras}
A \textbf{sequence of algebras under $\C\B$} is a triple
$(\A,\rho_\A,\iota_\A)$, where $(\A,\rho_\A)$ is a diagonal
$\C$-category under $\C\B$ and $\iota_\A:\underline{\A}\to
\underline{\A}_{+1}$ is a faithful morphism such that for all $n\in
\N$ the following diagram commutes:
$$\xymatrix{ \C\B_n\ar[r]\ar@{^{(}->}[d]^\iota &
\C\rho_\A(\B_n)\ar@{^{(}->}[d]^{\iota_\A} \\
 \C\B_{n+1}\ar[r] & \C\rho_{\A}(\B_{n+1})}$$

A sequence of algebras under $\C\B$ is called \textbf{unitary} if $\A$ is a
$C^*$-category and the morphisms $\rho_A$ and $\iota$ are $*$-functors.
\end{definition}
When the maps $\rho_\A$ and $\iota_\A$ are clear from the context we will often
just write $\A$ for a sequence of algebras under $\C\B$.

\begin{remarks}\label{seq of alg rems}
\begin{enumerate}
\item Clearly if $(\A,\rho_\A,\iota_\A)$ is a sequence of algebras under $\C\B$
then so is $(\underline{\A},\rho_\A,\iota_\A)$.
\item Suppose $(\A,\rho_\A,\iota_\A)$ is a sequence of algebras under $\C\B$ and
$(\vartheta_n,V_n)$ is a \textit{faithful}
  representation of $(\underline{\A},\rho_\A)$ so that
$\vartheta_n:\End_{\underline{\A}}([n])\cong
\End_{\underline{\vartheta(\A)}}([n])$ is invertible for all $n$.  Then
$(\vartheta(\A),\vartheta\circ \rho_\A,\vartheta_{n+1}\circ \iota_\A\circ
\vartheta_n^{-1})$ is a sequence of algebras under $\C\B$. Moreover, setting
$\rho_n=\vartheta_n\circ\rho_\A:\C\B_n\to\End_\C(V_n)$ we obtain a sequence of
braid representations $(\rho_n,V_n)$ (in the sense of Definition \ref{seq}) with
injective algebra homomorphisms
$$\tau_n=\vartheta_{n+1}\iota_\A\vartheta_n^{-1}
:\rho_n(\C\B_n)\hookrightarrow\rho_{n+1}(\C\B_{n+1})$$ satisfying
$\tau_n\circ\rho_n=\rho_{n+1}\circ\iota$.
  \item A sequence of braid representations $(\rho_n, V_n)$ as in Definition
\ref{seq} gives rise to a sequence of
algebras under $\C\B$ as follows: Let $\mS=\mS(\rho_n,V_n)$ denote the diagonal
category with $\End_\mS([n]):=\End_\C(V_n)$ and define
$\rho_\mS=\rho_n:\C\B_n\to\End_\mS([n])$.  The injective algebra maps $\tau_n$
in Definition \ref{seq} give us
$$\iota_\mS=\tau_n:\C\rho_n(\B_n)=\End_{\underline{\mS}}([n])\to\End_{\underline
{\mS}}([n+1])=\C\rho_{n+1}(\B_{n+1})$$ so that
$(\mS,\rho_\mS,\iota_\mS)$ is a sequence of algebras under $\C\B$, which is
unitary if and only if the $\B_n$-representation $V_n$ is unitary for all $n$.

\end{enumerate}

\end{remarks}

\subsection{Key examples}

Sequences of algebras under $\C\B$ arise naturally in a number of settings.  We
single out some important examples for later use.

\begin{ex}\label{YB algebras ex}
Let $\CC$ be a monoidal $\C$-linear category. Given a Yang-Baxter
operator $c$ on $V\in\CC$, let
$(YB_{(V,c)},\rho^{(V,c)},\iota)$ be the sequence of
algebras under $\C \B$ defined as follows:
\begin{enumerate}
\item $\End_{YB_{(V,c)}}([n])= \End_\CC(V^{\circledast n})$,
\item $\iota:\End_{YB_{(V,c)}}([n])\rightarrow \End_{YB_{(V,c)}}([n+1])$ is
defined by $\iota(f)= f\otimes\id_{V}$ and
\item $\rho^{(V,c)}_n: \C\B_n \to \End_{YB_{(V,c)}}([n])$ is defined by
$\rho^{(V,c)}_n(\sigma_i)=c_i$ where $c_i$ is as in eqn (\ref{rep from YBO}).
\end{enumerate}
\end{ex}

If $H$ is a quasi-triangular (quasi-)Hopf algebra and $V$ is any $H$-module then
$c$ is a
Yang-Baxter operator on $V$ in the monoidal category $\Rep(H)$.  Thus
$YB_{(V,c)}$ has the structure of a sequence of algebras under
$\C\B$.

\begin{ex}\label{bfc ex}
If $\CC$ is a braided fusion category with simple objects
$\OO(\CC)=\{\one=X_0,\ldots,X_{k-1}\}$ and $X\in\OO(\CC)$ then
the braiding $c$ is a Yang-Baxter operator on $X$.  The semisimplicity of $\CC$
implies that the sequence of algebras
$(YB_{(X,c)},\rho^{(X,c)},\iota)$ under $\C\B$ has a faithful
representation: $(\vartheta_n,W_n^X)$ where
$W_n^X:=\bigoplus_i\Hom_\CC(X_i,X^{\ot n})$ is the minimal faithful
$\End_\CC(X^{\ot n})$-module.  This gives rise to the sequence of
$\B_n$-representations of interest in \cite{RW}.

\end{ex}

\begin{ex}\label{braid quotient ex}
 Naturally $\C\B$ itself has the structure of a sequence of algebra under
$\C\B$.  One may construct more sequences of algebras under $\C\B$ as quotients
of $\C\B$ provided certain compatibility constraints are satisfied.  For example
we can define a sequence of algebras $\TL_0(q,\ell)$ under $\C\B$ from the
Temperley-Lieb algebras by setting
$\End_{\TL_0(q,\ell)}([n])=TL_n(q)/Ann(\tr_\ell)$ where $q=e^{2\pi i/\ell}$ and
$Ann(\tr_\ell)$ is the annihilator of the associated (Markov) trace form on
$TL_n(q)$.  The maps $\rho_{\TL_0(q,\ell)}$ and $\iota_{\TL_0(q,\ell)}$ come
from the compatibility of the surjection $\C\B_n\twoheadrightarrow
TL_n(q)/Ann(\tr_\ell)$ with $\iota:\B_n\to\B_{n+1}$.
  The sequence of algebras $\TL_0(q,\ell)$ is equivalent to the sequence
$(YB_{(X,c)},\rho^{(X,c)},\iota)$ obtained from the generating simple object $X$
(analogous to the vector representation of $\ssl_2$) in the braided fusion
category $\CC(\ssl_2,\ell)$.  Indeed, $\End_{\CC(\ssl_2,\ell)}(X^{\ot n})\cong
TL_n(q)/Ann(\tr_\ell)$.

\end{ex}

\subsection{$\CC$-localization}

\begin{definition}\label{c-loc}
Let $\A$ be a sequence of algebras under $\C\B$ and $\CC$ be a
monoidal $\C$-category. A \textbf{$\CC$-localization} of $\A$ is a
Yang-Baxter operator $c$ on $V\in\CC$ and a faithful morphism
 $\phi: \underline{\A}\to YB_{(V,c)}$ such that $\phi\circ
\rho_\A= \rho^{(V,c)}$.

A \textbf{unitary $\CC$-localization} of a unitary sequence of
algebras $\A$ under $\C\B$ is a $\CC$-localization $c$ on $V$ in a
$C^*$-tensor category $\CC$, such that $c$ is a unitary
isomorphism and $\phi$ is a $*$-functor.
\end{definition}

The sequences of algebras under $\C\B$ constructed from Yang-Baxter
operators in a category $\CC$ obviously are localizable over $\CC$
using the same Yang-Baxter operator. This localization shall be
called the \textbf{trivial localization} and this kind of
localizations are not relevant for us.

The following result gives some non-trivial localizations for
sequences of algebras associated to Yang-Baxter operators.

\begin{prop}\label{induced localization}
Let $\CC, \D$ be monoidal $\C$-linear categories and $F:\CC\to \D$ a
faithful strict monoidal functor. Then for every Yang-Baxter
operator $(V,\sigma)$  the pair $(F(V),F(\sigma))$ is Yang-Baxter
operator and defines a $\D$-localization of $YB_{(V,\sigma)}$.
\end{prop}

That Definition \ref{c-loc} is a generalization of Definition \ref{localdef} is
demonstrated in:
\begin{prop}\label{seq braid rep prop}
Let $(\rho_n,V_n)$ be a sequence of $\B_n$-representations (in the
sense of Definition \ref{seq}) and $(\mS,\rho_\mS,\iota_\mS)$ the associated
sequence of algebras under $\C\B$ as in Remarks \ref{seq of alg rems}. Then
$(\rho_n,V_n)$ is
localizable in sense of Definition \ref{localdef} if and only if
$(\mS,\rho_\mS,\iota_\mS)$ is Vec$_f$-localizable.
\end{prop}
\begin{proof} First suppose that the sequence of $\B_n$-representations
$(\rho_n,V_n)$
has localization $(V,c)$ as in Definition \ref{localdef}.  Then there
are injective maps
$$\phi_n:\End_{\underline{\mS}}([n])=\C\rho_n(\B_n) \to \End(V^{\ot
n})=\End_{YB_{(V,c)}}([n])$$ such that
$\phi_n\circ \rho_\mS=\rho^{(V,c)}_n$ for all $n\in \N$.  Thus $\phi_n$ defines
a faithful morphism $\phi:\underline{\mS}\to YB_{(V,c)}$ such that $\phi\circ
\rho_\mS=\rho^{(V,c)}$, \emph{i.e.}, a Vec$_f$-localization of
$(\mS,\rho_\mS,\iota_\mS)$.

Conversely, a Vec$_f$-localization of $(\mS,\rho_\mS,\iota_\mS)$ is a braided
vector space $(V^\prime,c^\prime)$ and a faithful morphism
$\phi:\underline{\mS}\to YB_{(V^\prime,c^\prime)}$
such that $\phi\circ \rho_\mS=\rho^{(V^\prime,c^\prime)}$, so
$\phi_n:\End_{\underline{\mS}}=\C\rho_n(\B_n)
\to \End_{YB_{(V^\prime,c^\prime)}}=\End({V^\prime}^{\ot n})$ define a
localization of $(\rho_n,V_n)$ in the sense of Definition \ref{localdef}.
\end{proof}

This motivates the following short-hand notation:
\begin{definition}
 A Vec$_f$-localization of $\A$ will be called a \textbf{localization} and a
(unitary) $Hilb_f$-localization of $\A$ will be called a \textbf{unitary
localization}.
\end{definition}
If $(V,c)$ is a localization of $\A$ we will refer to $\dim(V)$ as the
\textbf{dimension}
of the localization.
\begin{corollary}\label{cor faithful}
Let $(\A,\rho_\A,\iota_\A)$ be a localizable sequence of algebras under $\C\B$.
Then for any \textit{faithful} representation
$(\vartheta_n,V_n)$ of $\A$  the sequence of $\B_n$-representations
$(\vartheta_n\circ\rho_\A,V_n)$ is localizable in the sense of Definition
\ref{localdef}.
\end{corollary}\qed

We illustrate this result with the following example.
\begin{ex}
 Consider the sequence $(YB_{(X,c)},\rho^{(X,c)},\iota)$ constructed from an
object in a braided fusion category $\CC$ as in Example \ref{bfc ex}.  The
faithful representation $(\vartheta_n,W_n^X)$ of $YB_{(X,c)}$ yields a sequence
of $\B_n$-representations $(\vartheta_n\circ\rho^{(X,c)},W_n^X)$.  If
$YB_{(X,c)}$ has localization $(V,R)$ with faithful morphism
$\phi:\underline{YB_{(X,c)}}\to\End_\C(V^{\ot n})$ satisfying
$\phi\circ\rho^{(X,c)}=\rho^{(V,R)}$ then defining $\Phi_n=\phi\circ
(\vartheta_n\mid_{\rho^{(X,c)}(\C\B_n)})^{-1}$ gives us the required algebra
injections
$(\vartheta_n\circ\rho^{(X,c)})(\C\B_n)\stackrel{\Phi_n}{\to}\End_\C(V^{\ot n})$
so that $(\vartheta_n\circ\rho^{(X,c)},W_n^X)$ is localizable.
\end{ex}


\subsection{Quasi-localization}

\begin{definition}

Let $\A$ be a sequence of algebras under $\C\B$. A \textbf{quasi-localization}
of $\A$ is a $\overline{(V,a,c)}$-localization where $(V,a,c)$ is a
quasi-braided vector space.

\end{definition}
Again, $\dim(V)$ will be the \textbf{dimension} of the quasi-localization
$(V,a,c)$.

\begin{prop}\label{localization of quasi-Hopf and Hopf}
Let $H$ be a quasi-triangular (resp. quasi-Kac algebra) quasi-Hopf
algebra and $V$ an (resp. unitary) $H$-module.  Then the sequence of
(resp. unitary) algebras $YB_{(V,c)}$ under $\C\B$ has a (resp.
unitary) quasi-localization of dimension $\dim(V)$.  Moreover, if
$H$ is a (resp. Kac algebra) Hopf algebra, the sequence of (resp.
unitary) algebras $YB_{(V,c)}$ under $\C\B$ has a (resp. unitary)
localization of dimension $\dim(V)$.
\end{prop}
\begin{proof}
Let $(V,a,c)$ be the quasi-braided vector space defined in Example
\ref{exam quasi hopf quasi braid}. Then the monoidal category
$\langle V \rangle\subset \Rep(H)$ is a full tensor subcategory of
$\overline{(V,a,c)}$ so by Proposition \ref{induced localization}
$(V,a,c)$ is quasi-localization of $YB_{(V,c)}$. If $H$ is a quasitriangular
quasi-Kac algebra
$(V,a,c)$ is a unitary quasi-braided vector space and
$\overline{(V,a,c)}$ is a $C^*$-tensor category, so the same
argument shows that $(V,a,c)$ is a unitary quasi-localization of
$YB_{(V,c)}$.

Finally, if  $H$ is a Hopf algebra (resp. Kac algebra) we can choose
$a$ trivial, so $(V,c)$ is a localization (resp. unitary localization).
\end{proof}

\begin{remark}\label{crit rem}
\begin{enumerate}
 \item Observe that if $\CC$ is an integral braided fusion category then there
is a
quasi-Hopf algebra $H$ such that $\CC\cong\Rep(H)$ by \cite[Theorem 8.33]{ENO}.
In general it is difficult to determine if $\CC$ actually has a fiber functor,
i.e. $H$ can be chosen to be a (coassociative) Hopf algebra.  Prop.
\ref{localization of quasi-Hopf and Hopf} shows that if the braid group
representation associated with $X\in\CC$ does not have a localization $(V,c)$
with $\dim(V)=\FPdim(X)$ then no fiber functor can exist.
\item Note that Prop. \ref{localization of quasi-Hopf and Hopf} holds for
\emph{topological} quasi-triangular quasi-Hopf algebras (see \cite{Kas}): we may
allow
the universal $R$-matrix to reside in a completion of $H\ot H$ as long as
it gives rise to a quasi-braided vector space $(V,a,c)$.
\end{enumerate}

\end{remark}
In analogy with Prop. \ref{seq braid rep prop} we use the following:
\begin{nota}
 Let $(\rho_n,V_n)$ be a sequence of braid representations and $\mS$ the
associated sequence of algebras under $\C\B$.  A quasi-localization $(V,a,c)$ of
$\mS$ will be also called a quasi-localization of $(\rho_n,V_n)$.
\end{nota}

\subsection{Weak localization}
In analogy with weak Hopf algebras we make the following:
\begin{definition}
Let $\A$ be a sequence of algebras under $\C\B$. A \textbf{weak localization}
of $\A$ is a Bimod$(R)$-localization of $\A$ where $R$
is a finite dimensional semisimple $\C$-algebra.
\end{definition}

Any (multi-)fusion category is equivalent to the representation category of a
weak Hopf algebra (\cite[Corollary 2.22]{ENO}).  We have the related:
\begin{prop}
Let $\CC$ be a fusion category and $c$ a Yang-Baxter operator on
$V\in\CC$. The sequence of algebras $YB_{(V,c)}$ has a weak-localization.
\end{prop}
\begin{proof}
By \cite{Ost} every fusion category admits a faithful exact monoidal
functor from $\CC$ to Bimod$(R)$ for some $R$, so the proposition follows from
Proposition \ref{induced localization}.
\end{proof}
\begin{remarks}
\begin{enumerate}
  \item A weak localization with $R$ a simple algebra, defines a localization
in Vec$_f$. In fact, since $R$ is simple, Bimod$(R)$ is monoidally
equivalent to Vec$_f$.

\item If $\CC$ is a fusion category, by \cite{Ost} there is a bijective
correspondence between structures of $\CC$-module categories and
faithful exact monoidal functors from $\CC$ to Bimod$(R)$ for $R$ a
semisimple algebra.
\end{enumerate}

\end{remarks}

\subsection{Generalized localization}
\begin{ex}\label{gYBseq}
Given a $(k,m)$-generalized braided vector space $(V,c)$, we define
$(gYB_{(V,c)},\rho^{(V,c)},\iota)$ to be the sequence of
algebras under $\C \B$ defined as follows:
\begin{enumerate}
 \item $\End_{gYB_{(V,c)}}([n])=
\End(V^{\otimes k+m(n-2)})$
\item $\iota:\End_{gYB_{(V,c)}}([n])\to
\End_{gYB_{(V,c)}}([n+1])$ is defined by $f\mapsto f\otimes\id_{V}^{\ot m}$ and
\item $\rho^{(V,c)}_n: \C\B_n \to \End_{gYB_{(V,c)}}([n])$ is defined by
$\sigma_i \mapsto c_i^{k,m}$ where $c_i^{k,m}$ is given by eqn. (\ref{rep from
gYB oper}).
\end{enumerate}

\end{ex}

\begin{definition}
Let $\A$ be a sequence of algebras under $\C\B$.  A
\textbf{$(k,m)$-localization of $\A$} is a $(k,m)$-generalized
braided vector space $(V,c)$ and a faithful morphism
 $\phi: \underline{\A}\to gYB_{(V,c)}$ such that $\phi\circ
\rho_\A= \rho^{(V,c)}$.

A \textbf{unitary $(k,m)$-localization} of a unitary sequence of
algebras under $\C\B$ is a $(k,m)$-localization $(V,c)$ over the
$C^*$-tensor category $Hilb_f$, such that $c$ is a unitary isomorphism
and $\phi$ is a $*$-functor.
\end{definition}
\begin{remark}
 Notice that one may define $(k,m)$-localizations over \textit{any} monoidal
category $\CC$ by modifying Example \ref{gYBseq}.  We do not do so as we are
ultimately interested in matrix representations of the braid groups.
\end{remark}
As in the case of quasi-localizations we use the following:
\begin{nota}
 Let $(\rho_n,V_n)$ be a sequence of braid representations and $\mS$ the
associated sequence of algebras under $\C\B$.  A $(k,m)$-localization $(V,c)$ of
$\mS$ will be also called a $(k,m)$-localization of $(\rho_n,V_n)$.
\end{nota}

\section{Unitarity of weakly group-theoretical fusion
categories}\label{unitarity of wgt}

\subsection{The dual of a $\CC$-module $*$-category}

\begin{definition}
A \textbf{weak $C^*$-Hopf algebra} (resp. a \textbf{quasi-$C^*$-Hopf algebra})
is a
weak Hopf algebra $(H,m,\Delta,\varepsilon)$ (resp. a quasi-Hopf
algebra $(H,m,\Delta,\varepsilon,\Phi, S)$), such that $H$ is a
finite dimensional $C^*$-algebra and $\Delta$ is a $*$-homomorphism
(resp. $\Delta$ is a $*$-homomorphism and $\Phi^*=\Phi^{-1}$).
\end{definition}
\begin{remark}
The uniqueness of the unit, counit and the antipode for weak Hopf
algebras (see \cite[Proposition 2.10]{BNSz}) imply that \[1^*=1,\  \
\ \varepsilon(x^*)=\overline{\varepsilon(x)},\  \  \  (S\circ
*)^2=\id_H.\]

The dual $\widehat{H}$ of a weak $C^*$-Hopf algebra is again a
$C^*$-algebra with $*$-operation (see \cite[Theorem 4.5]{BNSz})
\[\langle\phi^*, x\rangle= \langle\phi, \overline{S(x)}\rangle,   \
\  \  \text{for all } \phi\in \widehat{H},  x\in H.\]
\end{remark}

A \textbf{$*$-representation} of a weak $C^*$-Hopf algebra $H$ is a finite
dimensional Hilbert space $(V, \langle,\rangle_V)$ carrying a left
action of $H$, such that  $\langle u, x \cdot v\rangle_V = \langle x^*
 \cdot u, v\rangle_V$ for all $u, v \in V$ and $x \in H$. The morphisms from
$(V, \langle,\rangle_V)$ to $(W, \langle,\rangle_W)$ are defined to
be the $H$-module morphisms from $V$ to
$W$. The category so obtained will be denoted by
$\mathcal{U}$-$\Rep(H)$, and it is a unitary (multi)-fusion
category, see \cite[Section 3]{BSz}.
\begin{theorem}\label{End* equivalent to End}
Let $\CC$ be a unitary fusion category and $\M$ be a $\CC$-module
$*$-category with $\M$ indecomposable. Then the monoidal category
$\End_{\CC}^{*}(\M)$ is a unitary fusion category monoidal
equivalent to $\End_{\CC}(\M)$.
\end{theorem}
\begin{proof}
Let $R$ be a finite dimensional $C^*$-algebra such that
$\mathcal{U}$-$\Rep(R)$ is unitarily equivalent to $\M$. By
\cite[Propostion 3]{Ost} the $\CC$-module structure on $\M$ defines
an $R$-fiber functor $F:\CC\to $Bimod$(R)$, so by \cite[Theorem
4]{Ost} there is a canonical weak Hopf algebra $H$ such that
$\Rep(H)$ is monoidally equivalent to $\CC$ and $\Rep(\widehat{H})$ is
monoidally equivalent to $\End_{\CC}(\M)$. Since $\M$ is a
$\CC$-module $*$-category the $R$-fiber functor is a $*$-monoidal
functor, and $H$ is a weak $C^*$-Hopf algebra. The dual weak Hopf
algebra $\widehat{H}$ is again a weak $C^*$-Hopf algebra, so
$\End_{\CC}(\M)$ is a unitary fusion category. Finally, in order to
prove that $\End_{\CC}^{*}(\M)$ is $*$-monoidally equivalent to
$\End_{\CC}(\M)$, we can use again \cite[Theorem 4]{Ost} and that
every $\widehat{H}$-module is isomorphic to a unitary
$\widehat{H}$-module.
\end{proof}
\subsection{Tensor products of module categories}
In this section we shall define the tensor product of module
$*$-categories over unitary fusion categories, following
\cite{ENO3}. We shall denote the tensor product of $\CC$-module
categories defined in \textit{loc. cit.} by $\boxtimes_\CC$.

\smallbreak

\begin{definition}\label{ext product}
Let $\M_1$ and $\M_2$ be $*$-categories. The exterior tensor
product $\M_1\overline{\boxtimes}\M_2$ is the $*$-category with
following objects and morphisms:
$$\text{Obj}(\M_1\overline{\boxtimes}\M_2)=\{
\bigoplus_{i\in\mathcal{I}} X_i\overline{\boxtimes}Y_i:  \  X_i\in
\text{Obj}(\M_1),
Y_i\in \text{Obj}(\M_2), |\mathcal{I}|<\infty\},$$
$$\Hom_{\M_1\overline{\boxtimes}\M_2}(\bigoplus_{i\in\mathcal{I}}
X_i\overline{\boxtimes}Y_i,
\bigoplus_{i\in\mathcal{I}}
X_i'\overline{\boxtimes}Y_i')=\bigoplus_{i,j\in\mathcal{I}}\Hom_{\M_1}(X_i,
X_j')\otimes_{ \mathbb
C}\Hom_{\M_2}(Y_i,Y_j'),$$ and $*$-structure $(f\overline{\boxtimes} g)^*=
f^*\overline{\boxtimes} g^*$.
\end{definition}

\medbreak

Let $\CC, \D$ be unitary fusion categories, so that the $*$-category
$\CC\overline{\boxtimes}\D$ has an obvious $*$-fusion category structure. By
definition, a $(\CC,\D)$-bimodule $*$-category is a module category
over the unitary fusion category
$\CC\overline{\boxtimes}\D^{\text{rev}}$, where $\D^{\text{rev}}$ is
 $\D$ with reversed tensor product.

For the following two definitions let $\A$ be a $*$-category and $\M,
\mathcal{N}$ left and right
(strict) $\CC$-module $*$-categories, respectively.

\begin{definition}\cite[Definition 3.1]{ENO3}
Let $F :\M\overline{\boxtimes} \mathcal{N}\to \A$ be an exact
$*$-functor. We say that $F$ is \textbf{$\CC$-balanced} if there is a natural
family of unitary isomorphisms
\[b_{M,X,N} : F(M \otimes X\overline{\boxtimes} N) \to F(M\overline{\boxtimes}
X \otimes N),\] such that
\[b_{M,X\otimes Y,N}= b_{M,X,Y \otimes N}\circ b_{M\otimes X,Y,N},\]
for all $M\in \M, N\in \mathcal{N}, X,Y\in \CC$.
\end{definition}

\begin{definition}\cite[Definition 3.3]{ENO3}
A tensor product of a right $\CC$-module $*$-category $\M$ and a
left $\CC$-module $*$-category $\mathcal{N}$ is a $*$-category
$\M\overline{\boxtimes}_{\CC}\mathcal{N}$ together with a
$\CC$-balanced $*$-functor
 \[B_{\M,\mathcal{N}} :\M\overline{\boxtimes} \mathcal{N} \to
 \M\overline{\boxtimes}_\CC \mathcal{N}\]
inducing, for every $*$-category $\A$, an equivalence between the
category of $\CC$-balanced $*$-functors from $\M\overline{\boxtimes}
\mathcal{N}$ to $\A$ and the category of $*$-functors from
$\M\overline{\boxtimes}_\CC \mathcal{N}$ to $\A$:
\[\Hom_{bal}^*(\M\overline{\boxtimes}  \mathcal{N}, \A) \cong
\Hom^*(\M\overline{\boxtimes}_\CC \mathcal{N}, \A).\]
\end{definition}

\begin{remark}
\begin{enumerate}
   \item  The existence of the tensor product for module categories over
$*$-fusion
categories can be proved using the same ideas in \cite{ENO3}.
semisimple category, see \cite{ENO3}.
\end{enumerate}
\end{remark}

Given a right  $\CC$-module $*$-functor $F : \M \to \M'$ and a left
$\CC$-module *-functor $G : \mathcal{N} \to \mathcal{N}'$ note that
$B_{\M',\mathcal{N}'}(F \overline{\boxtimes} G) :
\M\overline{\boxtimes} \mathcal{N} \to \M' \overline{\boxtimes}_\CC
\mathcal{N}'$ is a $\CC$-balanced $*$-functor. Thus the universality
of $B$ implies the existence of a unique right $*$-functor $F
\overline{\boxtimes}_\CC G := \overline{B_{\M',\mathcal{N}'}(F
\overline{\boxtimes }G)}$ making the diagram

\[
\begin{diagram}
\node{\M\overline{\boxtimes} \mathcal{N}}
\arrow{s,l}{B_{\M,\mathcal{N}}} \arrow{e,l}{F\overline{\boxtimes} G}
\node{\M'\overline{\boxtimes}\mathcal{N}'}
\arrow{s,r}{B_{\M',\mathcal{N}'}}\\
\node{\M\overline{\boxtimes}_{\CC}\mathcal{N}}\arrow{e,t}{F\overline{\boxtimes}
_\CC
G} \node{\M'\overline{\boxtimes}_\CC \mathcal{N}'}
\end{diagram}
\]commutative.
The bihomomorphism $\overline{\boxtimes}_\CC$ is functorial in module
$*$-functors, \emph{i.e.}, $(F'\overline{\boxtimes}_\CC
E')(F\overline{\boxtimes}_\CC E)= F'F\overline{\boxtimes}_\CC E'E$.

\begin{remarks}\label{remarks sobre producto tensorial}

\begin{enumerate}
  \item If  $\M$ is a $(\CC,\mathcal{E})$-bimodule $*$-category and
$\mathcal{N}$
is an $(\mathcal{E},\D)$-bimodule $*$-category, then
$\M\overline{\boxtimes}_\mathcal{E}\mathcal{N}$ is a
$(\CC,\D)$-bimodule $*$-category and $B_{\M,\mathcal{N}}$ is a
$(\CC,\D)$-bimodule $*$-functor, see \cite[Proposition
3.13]{Justin}.
  \item Let $\M$ be a $(\CC,\D)$-bimodule $*$-category. The $\CC$-module action
on $\M$ defines
  a $\CC$-balanced $*$-functor. Let $r_\M :\CC\overline{\boxtimes}_\CC\M \to \M$
denote the unique
  $*$-functor factoring through $B_{\CC,\M}$. As in \cite[Proposition
3.15]{Justin} we can prove that $r_\M$ is
  a $(\CC,\D)$-module $*$-category equivalence.
  \item Let $\M$ be a right $\CC$-module $*$-category, $\mathcal{N}$ a
$(\CC,\D)$-bimodule
  $*$-category, and $\mathcal{K}$ a left $\D$-module $*$-category. Then
  as in  \cite[Proposition 3.15]{Justin}, there is a canonical equivalence
$(\M\overline{\boxtimes}_\CC
\mathcal{N})\overline{\boxtimes}_\mathcal{D} \mathcal{K} \cong
\M\overline{\boxtimes}_\CC (\mathcal{N} \overline{\boxtimes}_\D
\mathcal{K})$ of bimodule
*-categories. Hence the notation $\mathcal{M}\overline{\boxtimes}_\CC
\mathcal{N}\overline{\boxtimes}_\D
\mathcal{K}$ will yield no ambiguity.
\end{enumerate}
\end{remarks}

\subsection{Crossed product tensor categories}

We briefly recall group actions on tensor categories. For more
details the reader is referred to \cite{DGNO}.

\medbreak

Let $\CC$ be a tensor category and let
$\underline{\text{Aut}_\otimes(\CC)}$ be the monoidal category of
monoidal auto-equivalences of $\CC$, arrows are  tensor natural
isomorphisms and tensor product the  composition of monoidal
functors.

\medbreak

For any group $G$ we shall denote by $\underline{G}$ the monoidal
category where objects are elements of $G$ and tensor product is
given by the product of $G$. An action of the group  $G$ on
$\ca$, is a monoidal functor   $F:\underline{G}\to
\underline{\text{Aut}_\otimes(\CC)}$. In other words, for any
$\sigma\in G$ there is a monoidal functor $(F_\sigma,
\zeta_\sigma):\CC\to\CC$, and for any $\sigma,\tau\in G$, there are
natural monoidal isomorphisms $\gamma_{\sigma,\tau}:F_\sigma\circ
F_\tau\to F_{\sigma\tau}$.

Given an action $F:\underline{G}\to
\underline{\text{Aut}_\otimes(\CC)}$ of $G$ on $\CC$,  the
$G$-crossed product tensor category, denoted by $\CC\rtimes G$, is
defined as follows. As an abelian category $\CC\rtimes G=
\bigoplus_{\sigma\in G}\CC_\sigma$, where $\CC_\sigma =\CC$ as an
abelian category, the tensor product is
$$[X, \sigma]\otimes [Y,\tau]:= [X\otimes F_\sigma(Y),
\sigma\tau],\  \  \   X,Y\in \CC,\  \   \sigma,\tau\in G,$$ and the
unit object is $[1,e]$. See \cite{tambara} or \cite{Ga3} for the
associativity constraint and a proof of the pentagon identity.

\begin{lemma}\label{lema action unitaria}
Let $\CC$ be a unitary fusion category and $G$ a finite group acting
on $\CC$. Then $\CC\rtimes G$ is a unitary fusion category  if and
only if the $G$-action on $\CC$ is unitary, i.e., $F_\sigma$ are
monoidal $*$-functors and $\gamma_{\sigma,\tau}$ are unitary natural
isomorphisms for all $\sigma,\tau\in G$. If $\CC\rtimes G$ is a
unitary fusion category $\CC$ is a $\CC\rtimes G$-module
$*$-category.
\end{lemma}
\begin{proof}
Straightforward.
\end{proof}

\subsection{Clifford theory for $G$-graded fusion categories}

Let $G$ be a group and $\CC$ be a tensor category. We
shall say that  $\CC$ is  $G$-graded if there is a decomposition
$$\CC=\oplus_{\sigma\in G} \CC_\sigma$$ of $\CC$ into a direct sum of full
abelian subcategories, such that  for all $\sigma, \tau\in G$, the
bifunctor  $\otimes$ maps $\CC_\sigma\times \CC_\tau$ to
$\CC_{\sigma \tau}$. Given a $G$-graded tensor category $\CC$, and a
subgroup $H\subset G$, we shall denote by $\CC_H$ the tensor
subcategory $\bigoplus_{h\in H}\CC_h$.

\begin{definition}
Let $\CC$ be a $G$-graded fusion category. If
$(\M,\overline{\otimes})$ is a $\CC_e$-module category, then a
\textbf{$\CC$-extension of $\M$} is a $\CC$-module category $(\M,\odot)$ such
that $(\M,\overline{\otimes})$ is obtained by restriction to
$\CC_e$.
\end{definition}

\begin{prop}\label{proposition extention is *}
Let $\CC$ be a unitary fusion category graded by a group $G$ and
$(\M,\overline{\otimes})$ an indecomposable $\CC_e$-module
$*$-category. Then each $\CC$-extension $(\M,\odot)$ is a
$\CC$-module $*$-category.
\end{prop}
\begin{proof}Let $(\M,\odot)$ be a $\CC$-extension of
$(\M,\overline{\otimes})$. If $\overline{\M}=
\CC\overline{\boxtimes}_{\CC_e} \M$, then by \cite[Theorem 1.3]{Ga2}
$\End_\CC^*(\overline{\M})$ is a unitary fusion category
$*$-monoidally equivalent to a  $G$-semidirect product unitary fusion
category $\End_{\CC_e}^*(\M)\rtimes G$ and the $\CC$-extension
$(\M,\odot )$ is completely determined by the
$\End_{\CC_e}^*(\M)\rtimes G$-module category $\End_{\CC_e}^*(\M)$.
By Lemma \ref{lema action unitaria} $\End_{\CC_e}^*(\M)$ is a
$\End_{\CC_e}^*(\M)\rtimes G$-module $*$-category, thus the
$\CC$-extension is a $\CC$-module $*$-category.
\end{proof}

The following is a simplified version of Clifford theorem for
unitary fusion category, see \cite{Ga2}.
\begin{theorem}\label{clifford theo for *-categories}
Let $\CC$ be a $G$-graded $*$-fusion category, $\M$  an
indecomposable $\CC$-module $*$-category and $\mathcal{N}$ an
indecomposable $\CC_e$-submodule $*$-subcategory of $\M$. Then there
is a subgroup $S\subset G$ and a $\CC_S$-extension
$(\mathcal{N},\odot)$ of $\mathcal{N}$, such that $\M \cong
\CC\overline{\boxtimes}_{\CC_S}\mathcal{N}$ as $\CC$-module
$*$-categories.
\end{theorem}
\begin{proof}
It follows from \cite[Corollary 4.4]{Ga2} and Proposition
\ref{proposition extention is *}.
\end{proof}

\subsection{Completely unitary fusion categories}

\begin{definition}
Let $\CC$ be a fusion category. We shall say that $\CC$ is
\textbf{completely unitary} if the following
properties are satisfied:

\begin{enumerate}
  \item $\CC$ is monoidally equivalent to a unique (up to $*$-monoidal
equivalences)
  unitary fusion category (we shall
  denote this unitary fusion category again by $\CC$).
  \item Every  $\CC$-module category is equivalent to a unique
  (up to $\CC$-module  $*$-functor equivalences) $\CC$-module
  $*$-category.
  \item Every $\CC$-module functor equivalence between $\CC$-module
$*$-categories is
  equivalent to a unique (up to unitary $\CC$-module natural isomorphisms)
$\CC$-module
  $*$-functor equivalence.
\end{enumerate}
\end{definition}

\begin{remark}\label{coeficientes universales}

 Let
$U(1)=\{z\in \C: |z|=1\}$ and $G$ be a finite group. By the
universal coefficient theorem \cite[Theorem 10.22]{rot}
$H^n(G,U(1))=H^n(G, \C^*)$ for all $n>0$, \emph{i.e.}, every
$n$-cocycle with coefficients on $\C^*$ is equivalent to a some
$n$-cocycle with coefficients on $U(1)$.
\end{remark}
\begin{prop}\label{pointed fusion are complet unitary}
Every pointed fusion category is a completely unitary fusion category.
\end{prop}
\begin{proof}
It follows from Remark \ref{coeficientes universales} and the
classification of module categories over  pointed fusion categories,
\cite{O2}.
\end{proof}
 In \cite{ENO3} they show that a graded fusion
category $\CC=\bigoplus_{\sigma\in G}\CC_\sigma$ determines and it is
determined by the following data:

\begin{enumerate}
  \item a fusion category $\CC_e$, a collection of invertible $\CC_e$-bimodule
categories $\CC_\sigma, \sigma\in G$,
  \item a collection of $\CC_e$-bimodule isomorphisms
$M_{\sigma,\tau}:\CC_\sigma\boxtimes_{\CC_e} \CC_\tau\to \CC_{\sigma\tau}$,
  \item natural isomorphisms of $\CC_e$-bimodule functors
\[\alpha_{\sigma,\tau,\rho}:M_{\sigma,\tau
\rho}(\text{Id}_{\CC_\sigma}\boxtimes_{\CC_e} M_{\tau,\rho})\to M_{\sigma
\tau,\rho}(M_{\sigma,\tau}\boxtimes_{\CC_e} \text{Id}_{\CC_\rho})\] satisfying
the identity
\end{enumerate}
\begin{multline}\label{equacion datos graduada}
    M_{\sigma,\tau \rho k}(\id_\sigma\boxtimes_{\CC_e}\alpha_{\tau,\rho,k})\circ
\alpha_{\sigma,\tau \rho,k}(\text{Id}_{\CC_\sigma}
\boxtimes_{\CC_e}M_{\tau,\rho}\boxtimes_{\CC_e}\text{Id}_{\CC_k} )\\
= \alpha_{\sigma,\tau,\rho k}(\text{Id}_{\CC_\sigma}
\boxtimes_{\CC_e} \text{Id}_{\CC_\tau} \boxtimes_{\CC_e} M_{\rho,k})
\circ \alpha_{\sigma\tau,\rho,k}(M_{\sigma,\tau} \boxtimes_{\CC_e}
\text{Id}_{\CC_\rho} \boxtimes_{\CC_e} \text{Id}_{\CC_k}),
\end{multline}
for all $\sigma, \tau, \rho, k  \in  G$, where we use the notation
Id for the identity functor, and $\id$ for the identity morphism.

\begin{remark}
If $\CC$ is a $G$-graded fusion category where $\CC_e$ is a unitary
fusion category, then $\CC$ is a unitary fusion category with
$\CC_e$ as unitary fusion subcategory if and only if $\CC_\sigma$
are $\CC_e$-bimodule $*$-category, $M_{\sigma,\tau}$ are
$\CC_e$-bimodules $*$-functors, and $\alpha_{\sigma,\tau,\rho}$ are
unitary isomorphism.
\end{remark}

\begin{theorem}\label{inducion completely...}
If $\CC$ is a $G$-graded fusion category such that $\CC_e$ and
$\CC_e\overline{\boxtimes}\CC_e^{rev}$ are completely unitary then
$\CC$ and $\CC\overline{\boxtimes}\CC^{rev}$ are completely unitary.
\end{theorem}
\begin{proof}
First we shall show that $\CC$ is monoidally equivalent to a unique
unitary fusion category. Since
$\CC_e\overline{\boxtimes}\CC_e^{rev}$ is completely unitary for
each $\sigma \in G$, the $\CC_e$-bimodule category $\CC_\sigma$ is
equivalent to a unique $\CC_e$-bimodule $*$-category
$\overline{\CC}_\sigma$. The bifunctor
$\otimes:\CC_\sigma\times\CC_\tau\to \CC_{\sigma\tau}$ and the
complete unitarity of $\CC_e\overline{\boxtimes}\CC_e^{rev}$ define
for each pair $\sigma,\tau \in G$ a unique $\CC_e$-bimodule
$*$-functor
$M_{\sigma,\tau}:\overline{\CC}_\sigma\overline{\boxtimes}_{\CC_e}\overline{\CC}
_{\tau}\to
\overline{\CC}_{\sigma\tau}$, such that $$M_{f,gh}(Id_{\CC_f}
\overline{\boxtimes}_{\CC_e} M_{g,h}) \cong M_{f
g,h}(M_{f,g}\overline{\boxtimes}_{\CC_e} Id_{\CC_h}),$$ as
$\CC_e$-module functors. Now, using the polar decomposition (see
Remark \ref{remark polar}) and the associativity constraint of
$\CC$, there are unitary isomorphisms of $\CC_e$-module $*$-functors
$$\alpha_{\sigma,\tau,\rho}:M_{f,gh}(Id_{\CC_f} \overline{\boxtimes}_{\CC_e}
M_{g,h})
 \to M_{fg,h}(M_{f,g}\overline{\boxtimes}_{\CC_e} Id_{\CC_h}),$$
for all $\sigma, \tau,\rho\in G,$ such that the equation
\eqref{equacion datos graduada} holds. The  new $G$-graded fusion
category is equivalent to $\CC$ and it is a unitary fusion category.

Thus we may assume that $\CC$ is a unitary fusion category. Let
$\M$ be an indecomposable $\CC$-module category, then by the
complete unitarity of $\CC_e$ and  Theorem \ref{clifford theo for
*-categories}, $\M$ is equivalent to a $\CC$-module $*$-category. Moreover, if
$\M$ and
$\mathcal{N}$ are $\CC$-module $*$-categories equivalent as
$\CC$-module categories,  by \cite[Proposition 4.6]{Ga2}, Remark
\ref{coeficientes universales} and \cite[Theorem 1.3]{Ga2}, $\M$ and
$\mathcal{N}$ are equivalent as $\CC$-module $*$-categories and
every $\CC$-module equivalence is equivalent to a $\CC$-module
$*$-functor equivalence.

Finally, note that $\CC\overline{\boxtimes}\CC^{rev}$ is a $G\times
G^{op}$-graded fusion category where
$(\CC\overline{\boxtimes}\CC^{rev})_{(e,e)}=\CC_e\boxtimes\CC^{rev}_e$.
Thus by the second part of this proof
$\CC\overline{\boxtimes}\CC^{rev}$ is completely unitary.
\end{proof}
\subsubsection{Weakly group-theoretical fusion categories are
completely unitary} Let $\CC$ be an arbitrary fusion category. The
adjoint category $\CC_{ad}$ is the smallest fusion subcategory of
$\CC$ containing all objects $X \otimes X^*$, where $X \in \CC$ is
simple. There exists a unique faithful grading of $\CC$ for which
$\CC_e = \CC_{ad}$ (see \cite{GeNi}). It is called the universal
grading of $\CC$. The corresponding group is called the universal
grading group of $\CC$, and denoted by U$(\CC)$. All faithful
gradings of $\CC$ are induced by the universal grading, in the sense
that for any faithful grading U$(\CC)$ canonically projects onto the
grading group $G$, and $\CC_e$ contains $\CC_{ad}$.

Let $\CC$ be a fusion category. Let $\CC^{(0)} = \CC, \CC^{(1)} =
\CC_{ad}$ and $\CC^{(n)} = (C^{(n-1)})_{ad}$ for every integer
$n\geq 1$. The non-increasing sequence of fusion subcategories of
$\CC$
$$\CC = \CC^{(0)}\supseteq \CC^{(1)} \supseteq \cdots \supseteq
\CC^{(n)}\supseteq \cdots $$ is called the upper central series of
$\CC$.

\begin{definition}\cite{GeNi}
A fusion category $\CC$ is called \textbf{nilpotent} if its upper
central series converges to Vec$_f$; \emph{i.e.}, $\CC^{(n)} =$
Vec$_f$ for some $n$. The smallest number $n$ for which this happens
is called the nilpotency class of $\CC$.
\end{definition}
\begin{definition}\cite{ENO2}
A fusion category $\CC$ is called  \textbf{weakly group-theoretical}
if it is Morita equivalent to a nilpotent fusion category.
\end{definition}

\begin{theorem}\label{w g-t are unitary}
Every weakly group theoretical fusion category is a completely
unitary fusion category.
\end{theorem}
\begin{proof}
By Theorem \ref{End* equivalent to End}, we only need to prove that
every nilpotent fusion category is completely unitary.

Let $\CC$ be a nilpotent fusion category. We shall use induction on
the nilpotency class of $\CC$. If the nipotency class is one, then
$\CC$ is a pointed fusion category, so by Proposition \ref{pointed
fusion are complet unitary}, $\CC$ and $\CC\overline{\boxtimes}
\CC^{rev}$ are completely unitary fusion categories. Now, let $\CC$
be a nilpotent fusion category of nilpotency class $n$, so
$\CC_{ad}=\CC^{(1)}$ has $n-1$ nilpotency class and by hypothesis of
induction $\CC_{ad}$ and $\CC_{ad}\overline{\boxtimes}
(\CC_{ad})^{rev}$ are completely unitary, thus by Theorem
\ref{inducion completely...}, $\CC=\CC^{(0)}$ is a completely
unitary fusion category.
\end{proof}

A (weak or quasi)-Hopf algebra is called weakly group-theoretical if
$\Rep(H)$ is a weakly group-theoretical fusion category.

\begin{corollary}\label{semisimple Hop then Kac}
Every weakly group-theoretical (quasi)-Hopf algebra is isomorphic to
a (quasi)-Kac algebra.
\end{corollary}
\begin{proof}
Let $H$ be a weakly group-theoretical Hopf algebra. By Theorem
\ref{w g-t are unitary}, the fusion category $\Rep(H)$ is equivalent
to a unitary fusion category $\CC$, and the forgetful functor defines a
$\CC$-module structure over Vec$_f$, so again by the complete
unitarity of $\CC$ the fiber functor is equivalent to a unique exact
$*$-monoidal functor. By Tannaka-reconstruction theory for compact
quantum groups (see \cite{Wor}), the Hopf algebra $H$  associated to
a $*$-monoidal fiber functor is isomorphic to a finite dimensional
$C^*$-Hopf algebra, \emph{i.e.,} a Kac algebra.

Now suppose that $H$ is a weakly group-theoretical quasi-Hopf
algebra. Then $\Rep(H)$ is equivalent to a unique unitary fusion
category $\CC$. By \cite[Proposition 2.1]{Mueg} every unitary fusion
category is a $C^*$-tensor category, so for every pair of objects $X,Y$,
$\Hom(X,Y)$ is a Hilbert space and $\langle fg, h\rangle =\langle
g,f^*h\rangle$, $\langle fg,h\rangle =\langle f, hg^*\rangle$ for
all morphisms in $\CC$.

Let $R\in\CC$ be the regular object, then $F(X)=\Hom(R,X)$ defines a
$*$-functor $F:\CC\to $Hilb$_f$. For every pair of simple objects
$X_i, X_j \in \CC$ there is a unitary isomorphism
$J_{i,j}:F(X_i)\otimes_{\C}F(X_j)\to F(X_i\otimes X_j)$, that defines a
quasi-fiber functor preserving the $*$-structure with unitary
constraint. Then by a standard reconstruction argument the algebra
$\End_\C(F )$ of functorial endomorphisms of $F$ has a natural
structure of quasi-Kac algebra, such that
$\mathcal{U}$-$\Rep(\End_\C(F ))\cong \CC$ as unitary fusion
categories.
\end{proof}
\begin{remarks}
\begin{enumerate}
  \item An analogous result to Corollary \ref{semisimple Hop then
  Kac} is true for weakly group-theoretical weak Hopf algebras.
  \item In the survey article \cite{Andrus}, the following Question 7.8 is
raised.
Given a semisimple Hopf algebra $H$, does it admit a compact
involution?  Corollary \ref{semisimple Hop then Kac} gives an
affirmative answer for weakly group theoretical Hopf algebras.
\end{enumerate}
\end{remarks}
It is not known (see \cite{ENO2} Question 2) if there exist weakly integral
fusion categories that are not weakly group-theoretical.  Theorem \ref{w g-t are
unitary} inspires the following
\textbf{question}: \emph{Is every weakly integral fusion category
completely unitary or unitary?}

\begin{remarks}
\begin{enumerate}
  \item The answer is ``no" without the weak integrality condition. In fact,
every unitary fusion category is pseudo-unitary, (see \cite[Section
8.4]{ENO}) but for example the Yang-Lee category is a non-integral
and non pseudo-unitary fusion category.  Indeed unitarity can fail
in very dramatic ways, see \cite{Ribbon}.
  \item The question can be reduced to integral fusion categories.  By
\cite[Proposition 8.27]{ENO} and \cite{GeNi} for every weakly integral fusion
  category $\CC$ there is $G$-grading such that $\CC_e$ is an integral fusion
category, so by
  Theorem \ref{inducion completely...}, $\CC$ is completely unitary
  if and only if $\CC_e$ is completely unitary.
\end{enumerate}
\end{remarks}

If $\CC$ is a unitary fusion category, the \textbf{unitary center}
$\mathcal{Z}^*(\CC)$ is defined as the full fusion subcategory of
the usual center $\mathcal{Z}(\CC)$, where
$(X,c_{X,-})\in\mathcal{Z}^*(\CC)$ if $c_{X,W}:X\otimes W\to
W\otimes X$ are unitary natural transformations for all $W\in \CC$.
It is easy to see that $\mathcal{Z}^*(\CC)$ is a unitary fusion
category. The following result appears in \cite[Theorem
6.4]{Mueg2} we provide an alternate proof using our notation.

\begin{prop}\label{centro unitario}
Let $\CC$ be a unitary fusion category. Then $\mathcal{Z}^*(\CC)$ is
braided monoidally equivalent to $\mathcal{Z}(\CC)$.
\end{prop}
\begin{proof}
Let $\CC\overline{\boxtimes} \CC^{rev}$ be the external tensor product
with the obvious structure of $*$-fusion category, see Definition
\ref{ext product}. The $*$-category $\CC$ is a
$\CC\overline{\boxtimes} \CC^{rev}$-module $*$-category.  By
\cite[Proposition 2.2]{O2} the center is equivalent to
$\End_{\CC\overline{\boxtimes} \CC^{rev}}(\CC)$ and it is easy to
see that unitary center is $*$-monoidally equivalent to
$\End_{\CC\overline{\boxtimes} \CC^{rev}}^*(\CC)$, so by Theorem
\ref{End* equivalent to End} they are monoidally equivalent.
\end{proof}

\begin{corollary}\label{weakly gt unitarizable}
Let $\CC$ be a weakly group-theoretical braided fusion category.
Then for every object $X\in \CC$, the $\B_n$-representation on
$\End_{\CC}(X^{\ot n})$ is unitarizable.  Moreover, if $\FPdim(X)\in
\N$ then the sequence $YB_{(X,c)}$ of algebras under $\C\B$ has a
unitary quasi-localization of
dimension $\FPdim(X)$, and if $\CC$ has a fiber functor then
$YB_{(X,c)}$ has a unitary localization of dimension
$\FPdim(X)$.
\end{corollary}
\begin{proof}We shall prove that every braided weakly
group-theoretical fusion category is braided equivalent to a unitary
braided fusion category.

The fusion category $\CC$ is weakly group-theoretical so it is a completely
unitary fusion category. Since $\CC$ is braided, we have a canonical
injective braided monoidal functor $F : \CC \to \mathcal{Z}(\CC)$,
but by Proposition \ref{centro unitario}, $\mathcal{Z}(\CC)$ is
braided equivalent to the unitary center $\mathcal{Z}^*(\CC)$, so
$\CC$ is braided equivalent to a unitary braided fusion category,
\emph{i.e.}, the braiding maps are unitary.

The others parts of the corollary follow from Proposition
\ref{localization of quasi-Hopf and Hopf} and Corollary
\ref{semisimple Hop then Kac}.
\end{proof}

\subsection{$\CC(\mathfrak{sl}_3,6)$: A case study}\label{case
study}
In this subsection we investigate a particular sequence
of braid group representations that does not appear to have a
localization, but does have both quasi- and $(k,m)$-localizations.  This
illustrates the
criterion mentioned in Remark \ref{crit rem} and justifies the notion of
$(k,m)$-localizations.

The integral unitary modular category $\CC(\ssl_3,6)$ has $10$ simple objects
and $\FP$-dimension $36$.  It was shown in \cite{NR} that
$\CC(\ssl_3,6)$ is non-group-theoretical and in fact has minimal
dimension among non-group-theoretical integral modular
categories.  The integrality of $\CC(\ssl_3,6)$ implies
that there is a semisimple, finite dimensional quasi-triangular
quasi-Hopf algebra $A$ such that $\Rep(A)\cong\CC(\ssl_3,6)$ as braided fusion
categories.  The simple object $X$ analogous to the vector
representation of $\ssl_3$ has $\FPdim(X)=2$ and tensor-generates
$\CC(\ssl_3,6)$. By Jimbo's quantum Schur-Weyl duality (\cite{Ji}), the unitary
sequence of $\B_n$ representations $(\rho_X,W_n^X)$ is equivalent to
the Jones-Wenzl representations factoring over the semisimple quotients
$\mathcal{H}_n(3,6)$ of the Hecke-algebras $\mathcal{H}_n(q)$ with
$q=e^{2\pi i/3}$ (see \cite{Rquat}).  Explicitly, one has an
isomorphism $\End(X^{\ot n})\cong\mathcal{H}_n(3,6)$ which
intertwines the $\B_n$ representations, and $\End(X^{\ot n})$ is
generated by the image of the braid group.  The eigenvalues of
$\rho_X(\sigma_i)$ are $-1$ and $e^{2\pi i/6}$ in this case.

Moreover, $(\rho_X,W_n^X)$ has a $2$-dimensional unitary
quasi-localization by Corollary \ref{weakly gt unitarizable}.
However, to explicitly determine the quasi-braided vector space
$(V,a,c)$ would require solving the pentagon and hexagon
equations, a notoriously difficult task. The task would be
significantly easier if $\CC(\ssl_3,6)$ were equivalent to $\Rep(H)$
for some (strictly coassociative) Hopf algebra $H$ since then one may assume $a$
is trivial, \emph{i.e.}
$(\rho_X,W_n^X)$ would have a $2$-dimensional localization.  This is
not the case:

\begin{lemma}\label{no unitary loc}
There is no unitary braided vector space of the form $(V,c)$ with
$\dim(V)=2$ localizing $(\rho_X,W_n^X)$ (in the sense of Definition
\ref{localdef}).
\end{lemma}
\begin{proof}
Dye \cite[Theorem 4.1]{Dye} has classified all $4\times 4$ unitary
solutions to the Yang-Baxter equation.  Up to multiplying by a
scalar and conjugation by matrices of the form $Q\otimes Q$ the
solutions are of 4 forms.  To see that none of these can localize
$(\rho_X,W_n^X)$ one need only check that the eigenvalues are not of
the form $\{-\chi,\chi e^{2\pi i/6}\}$ with $\chi\in\C$.  This is accomplished
in
\cite{FRW,Jenn}.
\end{proof}

Applying our results we obtain the following:

\begin{theorem}
Any semisimple quasi-Hopf algebra $A$ with
$\CC(\ssl_3,6)\cong\Rep(A)$ is non-coassociative.
\end{theorem}
\begin{proof}

Since $\FPdim(\CC(\ssl_3,6))=36=2^23^2$, by \cite[Theorem
1.6]{ENO2} $\CC(\ssl_3,6)$ is solvable and hence weakly
group-theoretical. Thus by Corollary \ref{weakly gt unitarizable} if
$\CC(\ssl_3,6)$ admits a fiber functor $(\rho_X,W_n^X)$ admits a
unitary localization of dimension two and this contradicts Lemma
\ref{no unitary loc}.
\end{proof}

Set $\zeta=e^{2\pi i/8}$ and consider the $8\times 8$ unitary block-diagonal
matrix:

\begin{equation}\label{gYB}
R=\frac{-e^{-\pi i/3}}{\sqrt{2}}\begin{pmatrix}
    \zeta^{-1}&0&-\zeta^{-1}&0\\0 &\zeta &0&\zeta\\
\zeta&0&\zeta&0\\0&-\zeta^{-1}&0&\zeta^{-1}
   \end{pmatrix}\oplus \frac{1}{\sqrt{2}}\begin{pmatrix}
    \zeta&0&\zeta&0\\0 &\zeta^{-1} &0&-\zeta^{-1}\\
-\zeta^{-1}&0&\zeta^{-1}&0\\0&\zeta&0&\zeta
   \end{pmatrix}
\end{equation}
Then $(\C^2,R)$ is a $(3,1)$-generalized braided vector space.  In fact, we have
the
following:
\begin{theorem}
$(\C^2,R)$ with $R$ as in (\ref{gYB}) gives a
$(3,1)$-localization of the sequence of braid group representations
$(\rho_X,W_n^X)$.
\end{theorem}
\begin{proof}
By the discussion above we may replace $(\rho_X,W_n^X)$ by the
Jones-Wenzl representation $(\pi_n,V_n)$ associated with the
semisimple quotient $\mathcal{H}_n(3,6)$ of the specialized Hecke
algebra $\mathcal{H}_n(q)$ with $q=e^{2\pi i/6}$ (see
\cite{Wenzl88}).  Here the quotient is by the annihilator of the
trace $tr$ on $\mathcal{H}_n(q)$ uniquely determined by
\begin{enumerate}
 \item $tr(1)=1$
\item $tr(ab)=tr(ba)$
\item $tr(be_n)=tr(b)\eta$ where $b\in\mathcal{H}_{n-1}(q)$,
\end{enumerate}
where $e_n$ are the generators of $\mathcal{H}_n(q)$ and
$\eta=\frac{1-q^{-2}}{1+q^3}$.

Thus it is enough to show that the $\mathcal{B}_n$-representation
afforded by $R$ in (\ref{gYB}) factors over $\mathcal{H}_n(3,6)$ and
induces a faithful representation.  Direct calculation shows that
the $\rho_R(\sigma_i)$ indeed satisfy the defining relations of
$\mathcal{H}_n(q)$ for $q=e^{2\pi i/6}$.  Defining $Tr$ on
$\End({\C^2}^{\ot n+1})$ (here $k+m(n-2)=n+1$) as $\frac{1}{2^{n+1}}$
times the usual trace, one concludes that
$\rho_R^{-1}(Tr)(a):=Tr(\rho_R(a)$ coincides with $tr$ by checking
that it satisfies the relations above using standard techniques (see
eg. \cite[proof of Lemma 3.1]{Rquat}).  In particular we see that
the kernel of $\rho_R$ (restricted to $\mathcal{H}_n(q)$) lies in
the annihilator of the trace $tr$ so that $\rho_R$ factors over the
quotient $\mathcal{H}_n(3,6)$.  Faithfulness of the induced
representation follows from a dimension count or by observing that
$Tr$ is faithful on $\End((\C^2)^{\ot n})$.

\end{proof}

\begin{remarks}\begin{enumerate}
                \item  The matrix $R$ above was derived from unpublished notes
of
Goldschmidt and
Jones describing a sequence of quaternionic representations of $\B_n$.  Indeed,
let $\imath$ and $\jmath$ denote the usual generators of $\mathbb{H}$ the
quaternionic division algebra.  Define
$$r=-e^{-\pi
i/3}/2(1+\imath\otimes\jmath\otimes\imath+1\otimes(\imath\jmath)\otimes
1+\jmath\otimes\imath\otimes\imath)$$
as an element of $\mathbb{H}^{\ot 3}$.  Applying the faithful $2$-dimensional
representation of $\mathbb{H}$ to $r$ yields the matrix $R$ in (\ref{gYB}).
\item The category $\CC(\ssl_3,6)$ is a subquotient of $\Rep(U_q\ssl_3)$ with
$q=e^{\pi i/6}$ with the $2$-dimensional simple object $X$ corresponding to the
$3$-dimensional vector
representation $V$ of $U_q\ssl_3$.  As such, there is a $9\times 9$
matrix solution $\check{R}$
to the Yang-Baxter equation that associated to $V$ (due to Jimbo \cite{Ji}).
However,
$\check{R}$ is not unitary and the braid group representations it affords has
subrepresentations that are not subrepresentations of $(\rho_X,W_n^X)$.
\item The results of \cite{FKW} imply that the representation spaces
$W_n^X$ can be uniformly embedded in a ``local'' Hilbert space but the braid
group only acts on a small subspace.  This issue was the main motivation
for \cite{RW}.
               \end{enumerate}

\end{remarks}

\section{Conjectures and General Results}\label{conj and results}

We have the following version of \cite[Conjecture 4.1]{RW}:
\begin{conj}\label{localintegral}
Let $X$ be
a simple object in a fusion category $\CC$ with $(X,c)$ a
Yang-Baxter operator.  Then the following are equivalent:
\begin{enumerate}
 \item[(a)] The sequence of algebras
$(YB_{(X,c)},\rho^{(X,c)},\iota)$ under $\C\B$ has a (unitary) generalized
localization
\item[(b)] The sequence of algebras
$(YB_{(X,c)},\rho^{(X,c)},\iota)$ under $\C\B$ has a (unitary)
quasi-localization
\item[(c)] $\FPdim(X)^2\in\N$
\item[(d)] $\rho^{(X,c)}(\B_n)$ is a finite group for all $n$.
\end{enumerate}
\end{conj}

\begin{remarks}
\begin{enumerate}
\item Corollary \ref{weakly gt unitarizable} shows that if $\CC$ is weakly
group-theoretical and $\FPdim(X)\in\N$ then (b) above holds.  It is not known if
there are any integral (or even weakly integral) fusion categories that are not
weakly group-theoretical.  If indeed $\CC$ integral implies $\CC$ weakly
group-theoretical then $\FPdim(X)\in\N$ implies (b), giving a slightly weaker
version of (c)$\Rightarrow$(b).  For this reason we parenthesize the word
\emph{unitary} as it is not unreasonable to expect that the conjecture is true
with or without unitarity.
 \item Conjecture 4.1 of \cite{RW} is only concerned with simple objects in
braided fusion categories for which the sequence of representations
$(\rho_n,W_n^X)$ is (unitarily) localizable.  We are not aware of any
counterexample to this more restrictive conjecture, but if the
example discussed in Subsection \ref{case study} is such a
counterexample then Conjecture \ref{localintegral} is an appropriate
replacement.
\item If $\CC$ is group theoretical (and hence integral) it is known \cite{ERW}
that (d) holds for any object $X\in\CC$ (for example if $\CC\cong\Rep(D^\omega
G)$, where $D^\omega G$ is the twisted double of a finite group $G$).
Integrality of $\CC$ and Prop. \ref{localization of quasi-Hopf and Hopf} imply
that (b) and (c) also hold.  The equivalence (c)$\Leftrightarrow$(d) has been
considered elsewhere, see \cite{jones86,jones89,GJ,FLW,LRW,RUMA,Rquat,NR}.
\end{enumerate}

\end{remarks}

The following version of \cite[Theorem 4.5]{RW} verifies part of Conjecture
\ref{localintegral} under an additional
hypothesis:
\begin{theorem}\label{thm:main} Let $X$ be a simple object in a fusion category
$\CC$ with $c$ a Yang-Baxter operator on $X$ such that
\begin{enumerate}
 \item[(a)]  $\underline{YB_{(X,c)}}=YB_{(X,c)}$ (i.e. the image of $\B_n$
generates $\End(X^{\ot n})$ as an algebra), and
\item[(b)] the sequence of algebras $(YB_{(X,c)},\rho^{(X,c)},\iota)$ under
$\C\B$ has either a generalized or quasi-localization
\end{enumerate}
then $\FPdim(X)^2\in\N$.
\end{theorem}

\begin{proof}
By passing to fusion subcategories, we may assume $X$ is a tensor generator for
$\CC$.

For quasi-localization Corollary \ref{cor faithful} implies that the same proof
as in \cite[Theorem 4.5]{RW} goes through without change as quasi-localizability
implies the same combinatorial consequences as localizability.

For the case of generalized localization we adapt the proof in \cite[Theorem
4.5]{RW}.
Suppose $(W,\gamma)$ is a $(k,m)$-localization of
$(YB_{(X,c)},\rho^{(X,c)},\iota)$ with injective morphism
$\phi:\underline{YB_{(X,c)}}\to gYB_{(W,\gamma)}$.  Observe that
$\End_{gYB_{(W,\gamma)}}=\End_\C(W^{\ot k+m(n-2)})$ and set $w=\dim(W)$.  Fix an
ordering $X_0=\one,X_1,\ldots,X_{N-1}$ of the simple objects and define
$H_i^n=\Hom(X_i,X^{\ot n})$ for each $i$ with $\Hom(X_i,X^{\ot n})\neq 0$.  By
(a), the $H_i^n$ form a complete set of simple $\rho^{(X,c)}(\C\B_n)$-modules.
Let $N_X$ denote the fusion matrix of $X$, \emph{i.e.} such that
$(N_X)_{i,j}=\dim\Hom(X_j,X\ot X_i)$.  Let $l\in\N$ be minimal with the property
that for all $i=0,\ldots,N-1$ there exists an $s\leq l$ such that $H_i^s\neq 0$
($l$ exists by \cite[Lemma F.2]{DGNO}).  Now let $p\geq 1$ be minimal such that
$H_0^p\neq 0$ so that $N_X$ is an irreducible matrix of period $p$.  Denote by
$G_n$ the inclusion matrix for the algebras $\End(X^{\ot n})\subset \End(X^{\ot
n+1})$, \emph{i.e.} the matrix of multiplicities obtained by restricting the
$H_j^{n+1}$ to $\End(X^{\ot n})$ and decomposing into a direct sum of $H_i^{n}$.
Semisimplicity of $\End(X^{\ot n})$ and the injectivity of $\phi$ imply that
$W^{\ot k+m(n-2)}\cong \bigoplus_i \mu_i^n H_i^n$ as $\End(X^{\ot n})$-modules
for some $\mu_i^n>0$.  Define a vector of multiplicities $(\ba_n)_i:= \mu_i^n$.
The commutativity of the diagram in Definition \ref{seq of algebras} and the
injectivity of $\phi$ imply:
\begin{equation}\label{comb cons}
 w^m \mathbf{a}_n=G_n \mathbf{a}_{n+1}
\end{equation}
 for all $n\geq l$.

The (square!) inclusion matrices $G^{(i)}= \Pi^{p-1}_{j=0}G_{l+i+j}$ of
$\End(X^{\otimes l+i})\subset \End(X^{\otimes l+i+p})$ are primitive (as $N_X$
is irreducible of period $p$) and hence the Perron-Frobenius eigenvalue of each
$G^{(i)}$ is simple.
 We will first show that $\mathbf{a}_l$ is a Perron-Frobenius eigenvector for
$G:=G^{(0)}$. Eqn. (\ref{comb cons}) implies that
$$(w^m)^{p} \mathbf{a}_l=G\mathbf{a}_{l+p}.$$
For simplicity, let us define $\alpha_0:=\mathbf{a}_l$,
$\alpha_n:=\mathbf{a}_{l+pn}$ and $M=(w^m)^{p}$. In this notation, the above
equation implies
$$M^n \alpha_0=G^n \alpha_n$$
\noindent for all $n\geq 0$. Let $\Lambda$ denote the Perron-Frobenius
eigenvalue of $G$ and let $\mathbf{v}$ and $\mathbf{w}$ be positive right and
left eigenvectors such that $\mathbf{w} \mathbf{v}=1$ and $\lim_{s \rightarrow
\infty}(\frac{1}{\Lambda}G)^s=\mathbf{v} \mathbf{w}$. We can rewrite the above
equation as follows:
$$\alpha_0= \frac{\Lambda^n}{M^n} \left(\dfrac{1}{\Lambda} G \right)^n
\alpha_n.$$
By taking a limit on the right-hand-side and then applying the property of the
eigenvectors $\mathbf{v}$ and $\mathbf{w}$ we have:
$$\alpha_0= \lim_{n \rightarrow \infty}\frac{\Lambda^n}{M^n} \mathbf{v}
\mathbf{w} \alpha_n = \lim_{n \rightarrow \infty} \left(\frac{\Lambda^n
\mathbf{w} \alpha_n}{M^n} \right) \mathbf{v}.$$
The limit exists and thus $\alpha_0=\mathbf{a}_l$ is an eigenvector for $G$ as a
non-zero multiple of $\mathbf{v}$. Similarly, each $\alpha_j$ is also a
Perron-Frobenius eigenvector for $G$. The integrality of $G$ and $\alpha_0$
implies that $\Lambda$ is rational and, moreover, the eigenvalues of $G$ are
algebraic integers, thus $\Lambda$ is a (rational) integer.  The same argument
shows that each $\mathbf{a}_{l+i}$ is a Perron-Frobenius eigenvector for
$G^{(i)}$.
Now we will show that $\FPdim (X)^2 \in \mathbb{N}$. As $N_X$ is irreducible
with period $p$ we may reorder the simple objects $X_0,\ldots,X_{N-1}$ so that
$\left( N_X\right)^p$ is block diagonal with primitive blocks $G^{(i)}$.  Let us
denote the Perron-Frobenius eigenvalue of $N_X$ by $\lambda$. Then by the
Frobenius-Perron theorem each $G^{(i)}$ has $\lambda^p$ as its Perron-Frobenius
eigenvalue which is a (rational) integer by the above arguments.  But $\lambda$
must reside in an abelian (cyclotomic) extension of $\mathbb{Q}$ and this
implies that $\lambda^s \in \mathbb{Z}$ for some $s\leq 2$.
\end{proof}


\begin{thebibliography}{AA}
\bibitem[A]{Andrus} N. Andruskiewitsch,\emph{About finite dimensional Hopf
algebras.}
Quantum symmetries in theoretical physics and mathematics
(Bariloche, 2000), Contemp. Math., 294, 1�57, Amer. Math. Soc.,
Providence, RI , 2002.

\bibitem[AS]{AS} N. Andruskiewitsch and H.-J. Schneider,\emph{ Pointed Hopf
algebras.}
New directions in Hopf algebras,  1--68, Math. Sci. Res. Inst. Publ., 43,
Cambridge Univ. Press, Cambridge, 2002.


\bibitem[B]{Baez} J.\ Baez,
\emph{Higher-Dimensional Algebra II. 2-Hilbert Spaces}, Adv. Math.
\textbf{127} (1997), 125--189.

\bibitem[BK]{BK} B.\ Bakalov; A.\ Kirillov, Jr., {\em Lectures on
Tensor Categories and Modular Functors}, University Lecture Series,
vol.\ {\bf 21},  Amer.\ Math.\ Soc., 2001.




\bibitem[BNSz]{BNSz}G.\ B\"{o}hm, F.\ Nill and K.\ Szlach\'{a}nyi,
\emph{Weak Hopf algebras I. Integral theory and $C^*$-structure}, J.
Algebra, 221 (1999), 385--438.

\bibitem[BSz]{BSz}G.\ B\"{o}hm  and K.\ Szlach\'{a}nyi,
\emph{Weak Hopf Algebras: II. Representation Theory, Dimensions, and
the Markov Trace}, J. Algebra, 233 (2000), 156--212.


\bibitem[DGNO]{DGNO}  V. Drinfeld, S. Gelaki,  D. Nikshych and  V. Ostrik,
\emph{ On Braided Fusion Categories I}, Selecta Math. N.S.
\textbf{16}, 1 (2010) 1--119.




\bibitem[Dav]{Dav} A.\ Davydov, \emph{ Nuclei of categories with tensor
products}. Theory and Applications of Categories, Vol. 18, No. 16,
(2007), 440--472.



\bibitem[D]{Dye} H. Dye, \emph{Unitary solutions to the Yang-Baxter
equation in dimension four}. Quantum information processing
\textbf{2} (2002) nos. 1-2, 117--150 (2003).

\bibitem[ENO1]{ENO}P.\ Etingof, D.\ Nikshych and V.\ Ostrik,  \emph{On fusion
categories}. Ann. of Math. (2) 162 (2005), no. 2, 581--642.

\bibitem[ENO2]{ENO2}  P. Etingof,  D. Nikshych and  V. Ostrik,
\emph{Weakly group-theoretical and solvable fusion categories}, Adv.
Math. 226 (2011) 176--205.

\bibitem[ENO3]{ENO3}  P. Etingof,  D. Nikshych and  V. Ostrik,
\emph{Fusion categories and homotopy theory}, Quantum Topol. 1 (3)
(2010) 209--273. Preprint arXiv:0909.3140.

\bibitem[ERW]{ERW} P.\ Etingof, E.\ C.\ Rowell and S.\ J.\ Witherspoon,
\emph{Braid group representations from quantum doubles of finite
groups.} Pacific J. Math. \textbf{234}  (2008) no. 1, 33--41.


\bibitem[FRW]{FRW}  J.\ Franko, E.\ C.\ Rowell and Z.\ Wang, \emph{ Extraspecial
2-groups and images of braid group representations}, J. Knot Theory
Ramifications \textbf{15} (2006) no. 4, 1--15.

\bibitem[F]{Jenn} J. Franko, \emph{ Braid group representations
arising from the Yang-Baxter equation.} J. Knot Theory
Ramifications, \textbf{19} (2010), no. 4, 525--538.

\bibitem[FKW]{FKW}M.\ H.\ Freedman; A.\ Kitaev; Z.\ Wang,
\emph{Simulation of topological field theories by quantum
computers.} Comm. Math. Phys. \textbf{227} (2002) no. 3, 587--603.

\bibitem[FLW]{FLW} M.\ H.\ Freedman, M.\ J.\ Larsen and Z.\ Wang,
    \emph{The two-eigenvalue problem and density of Jones representation of
braid groups.}
    Comm.\ Math.\ Phys.\ {228} (2002), 177-199.


\bibitem[Ga1]{Ga1}  C. Galindo, \emph{Clifford theory for
tensor categories}, J. Lond. Math. Soc.,  (2) 83 (2011) 57--78.

\bibitem[Ga2]{Ga2}  C. Galindo, \emph{ Clifford theory for
graded fusion categories},  preprint \texttt{arXiv:1010.5283 }.

\bibitem[Ga3]{Ga3}  C. Galindo, \emph{Crossed product tensor categories},
J. Algebra, 337 (2011) 233�-252.


\bibitem[GeNik]{GeNi}  S. Gelaki and  D. Nikshych, \emph{Nilpotent fusion
categories},
Adv. Math.,  \textbf{217} (2008), 1053-1071.

\bibitem[GJ]{GJ} D.\ M.\ Goldschmidt and V.\ F.\ R.\ Jones, \emph{Metaplectic
link invariants.}
    Geom.\ Dedicata \textbf{31} (1989)  no.\ 2, 165--191.

\bibitem[GHJ]{GHJ} F. M.
Goodman, P. de la Harpe, V. F. R. Jones, Coxeter graphs and towers
of algebras. Mathematical Sciences Research Institute Publications,
14. Springer-Verlag, New York, 1989. x+288 pp.


\bibitem[Gr]{Justin} J. Greenough,
\emph{Monoidal 2-structure of Bimodule Categories},  J. Algebra, 324 (2010),
1818-1859. Preprint arXiv:0911.4979.


\bibitem[Ji]{Ji} M. Jimbo,
\emph{A q-analogue of U(gl(N+1)), Hecke algebra, and the Yang-Baxter equation.}
Lett. Math. Phys. \textbf{11} (1986), no. 3, 247–252.


\bibitem[J1]{jones86} V.\ F.\ R.\ Jones, \emph{Braid groups, Hecke algebras and
type ${\rm II}\sb 1$ factors,}
Geometric methods in operator algebras (Kyoto, 1983), 242--273,
Pitman Res.\ Notes Math.\ Ser.\ {\bf 123}, Longman Sci. Tech.,
Harlow, 1986.


\bibitem[J4]{jones89} V.\ F.\ R.\ Jones, \emph{On a certain value of the
Kauffman polynomial,} Comm. Math. Phys. \textbf{125} (1989), 459-467.

\bibitem[JS]{JS} A. Joyal and R. Street,
\emph{Braided Tensor Categories}, Adv. Math. 102, (1993), 20-78.

\bibitem[Ks]{Kas} C.\ Kassel, Quantum Groups. Graduate Texts in Mathematics
\textbf{155}, Springer-Verlag, New York, 1995.

\bibitem[LRW]{LRW}  M.\ J.\ Larsen, E.\ C.\ Rowell and Z.\ Wang:
    \emph{The $N$-eigenvalue problem and two applications.}
    Int.\ Math.\ Res.\ Not. \textbf{2005} (2005) no.\ 64, 3987--4018.


\bibitem[Mu1]{Mueg} M.\ Mueger,
\emph{Galois theory for braided tensor categories and the modular
closure}, Adv. Math. \textbf{150} (2000), 151--201.

\bibitem[Mu2]{Mueg2} M.\ Mueger,
\emph{ From subfactors to categories and topology II. The quantum
double of tensor categories and subfactors},J. Pure Appl. Alg.
\textbf{180} (2003), 159--219.



\bibitem[NR]{NR} D.\ Naidu and E.\ C.\ Rowell, \emph{A finiteness property
for braided fusion categories}, to appear in Algebr. Represent.
Theory.

\bibitem[O1]{Ost} V.\ Ostrik, \emph{ Module categories, Weak Hopf Algebras and
Modular invariants.},  Transform. Groups.\ \textbf{8} (2003), 177--206.

\bibitem[O2]{O2}  V. Ostrik, \emph{Module categories over the Drinfeld double of
a
    finite group}, Int.\ Math.\ Res.\ Not. (2003) no.\ 27, 1507-1520.

\bibitem[Rot]{rot}  J. Rotman, An Introduction to Homological Algebra, Academic
Press, New York-San Francisco-London, (1979).

\bibitem[R1]{Ribbon} E.\ C.\ Rowell \emph{On a family of non-unitarizable ribbon
categories}, Math. Z. 250 (2005), no. 4, 745--774.


\bibitem[R2]{RUMA} E.\ C.\ Rowell, \emph{Braid representations from quantum
groups of exceptional Lie type},
Rev. Un. Mat. Argentina \textbf{51} (2010) no. 1, 165--175.

\bibitem[R3]{Rquat} E.\  C.\ Rowell, \emph{A quaternionic braid representation
(after Goldschmidt and Jones)}, to appear in Quant. Topol. arXiv:1006.4808.

\bibitem[RSW]{RSW} E.\ Rowell, R.\ Stong and Z.\ Wang, \emph{ On classification
of modular tensor categories},  Comm. Math. Phys.  \textbf{292}  (2009),  no. 2,
343--389.

\bibitem[RW]{RW} E.\ C.\ Rowell and Z.\ Wang, \emph{Localization of unitary
braid group representations}, preprint arXiv:1009.0241

\bibitem[RZWG]{RZWG} E.\ C.\ Rowell, Y.\ Zhang, Y.-S.\ Wu and M.-L.\ Ge,
\emph{Extraspecial two-groups, generalized Yang-Baxter equations and braiding
quantum gates}, {Quantum Inf. Comput.} \textbf{10} (2010) no. 7-8, 0685--0702.

\bibitem[Tam]{tambara}   D. Tambara, \emph{Invariants and semi-direct products
for finite group
    actions on tensor categories}, J. Math. Soc. Japan \textbf{53} (2001),
429--456.



\bibitem[TW]{TurWen97} V.\ G.\ Turaev; H.\ Wenzl, \emph{Semisimple and modular
tensor categories from link invariants}, Math. Ann. \textbf{309}
(1997), 411-461.


\bibitem[W1]{Wenzl88} H.\ Wenzl, \emph{Hecke algebras of type $A_n$ and
subfactors.}  Invent. Math.  \textbf{92}  (1988)  no. 2, 349--383.

\bibitem[Wo]{Wor} S.L\ Woronowicz, \emph{Tannaka-Krein duality for compact
matrix
pseudogroups. Twisted SU(N) groups.} Invent. Math. \textbf{93},
 (1988) 35-76.

\end{thebibliography}
\end{document}